\documentclass{amsart}

\usepackage[latin1]{inputenc}
\usepackage{harvard}
\usepackage{graphicx}
\usepackage{psfrag}
\usepackage{amsthm,amsfonts,amssymb,amscd,amsmath}

\newtheorem{theorem}{Theorem}[section]
\newtheorem{mainthm}{Main Theorem}
\newtheorem{oldthm}{Theorem}
\newtheorem{oldlemma}[oldthm]{Lemma}
\newtheorem{lemma}[theorem]{Lemma}
\newtheorem{prop}[theorem]{Proposition}
\newtheorem{corollary}[theorem]{Corollary}
\newtheorem{remark}[theorem]{Remark}

\newtheorem{definition}[theorem]{Definition}

\newcommand{\field}[1]{\mathbb{#1}}

\def\C{\mathbb{C}}

\def\fC{\field{C}}

\def\fN{\field{N}}

\def\fR{\field{R}}

\def\fZ{\field{Z}}

\def\cA{\mathcal{A}}
\def\cB{\mathcal{B}}
\def\cC{\mathcal{C}}
\def\cD{\mathcal{D}}
\def\cE{\mathcal{E}}

\def\cH{\mathcal{H}}

\def\cJ{\mathcal{J}}
\def\cL{\mathcal{L}}

\def\cM{\mathcal{M}}

\def\cR{\mathcal{R}}

\def\cP{\mathcal{P}}
\def\cS{\mathcal{S}}
\def\cU{\mathcal{U}}
\def\cW{\mathcal{W}}
\def\cV{\mathcal{V}}
\def\cX{\mathcal{X}}

\def\cY{\mathcal{Y}}

\def\bLambda{{\bf \Lambda_*}}
\def\blambda{{\boldsymbol \lambda_*}}
\def\bmu{{\boldsymbol \mu_*}}
\def\bF{{\bf F_*}}
\def\bG{{\bf G_*}}

\def\bs{{\bf s^*}}

\def\bT{{\bf T}}
\def\bW{{\bf W}}

\def\loc{{\rm loc}}

\newcommand{\converge}[1]{\stackrel{#1}{\longrightarrow}}
\def\HkG{{H_{k,G}}}
\def\iHkG{{\!\!\phantom{a}_iH_{k,G}}}
\def\hiHkG{{\!\!\phantom{a}_{\hat{i}}H_{k,G}}}

\newcommand{\bTi}[1]{{\bf T}_{[#1]}}
\newcommand{\bTis}[3]{{\bf T}_{[#1]_{#2}^{#3}}}
\newcommand{\TiG}[1]{T_{[#1],G}}
\newcommand{\Ti}[1]{T_{[#1]}}
\newcommand{\TisG}[3]{T_{[#1]_{#2}^{#3},G}}
\newcommand{\Tis}[3]{T_{[#1]_{#2}^{#3}}}

\begin{document}
\title[Dynamics of the Universal Area-Preserving Map: Stable Sets]{Dynamics of the Universal Area-Preserving Map Associated with Period Doubling: Stable Sets}
\author{Denis Gaidashev, Tomas Johnson}
\address{Department of Mathematics, Uppsala University, Box 480, 751 06 Uppsala, Sweden}
\email{{gaidash@math.uu.se},  {tomas.johnson@math.uu.se}}
\begin{abstract}
It is known that the famous Feigenbaum-Coullet-Tresser period doubling universality has a counterpart for area-preserving maps of ${\fR}^2$. A renormalization approach has been used in \cite{EKW1} and \cite{EKW2} in a computer-assisted proof of existence of a ``universal'' area-preserving map $F_*$ ---  a map with orbits of all binary periods $2^k, k \in \fN$.  In this paper, we consider {\it infinitely renormalizable} maps --- maps on the renormalization stable manifold in some neighborhood of  $F_*$ --- and study their dynamics. 

For all such infinitely renormalizable maps in a neighborhood of the fixed point $F_*$ we prove the existence of a ``stable'' invariant Cantor set  $\cC^\infty_F$ such that the Lyapunov exponents of $F \arrowvert_{\cC^\infty_F}$ are zero, and whose Hausdorff dimension satisfies
$${\rm dim}_H(\cC_F^{\infty}) < 0.5324.$$  

We also show that there exists a submanifold, $\bW_\omega$, of finite codimension in the renormalization local stable manifold, such that for all $F\in\bW_\omega$ the set $\cC^\infty_F$ is  ``weakly rigid'': the dynamics of any two maps in this submanifold, restricted to the stable set $\cC^\infty_F$, is conjugated by a bi-Lipschitz transformation that preserves the Hausdorff dimension.
\end{abstract}

\maketitle
\tableofcontents


\section{Introduction}

Universality --- independence of the quantifiers of the geometry of orbits and bifurcation cascades in families of maps of the choice of a particular family  ---  has been abundantly observed in area-preserving maps, both as the period-doubling universality \cite{DP,Hel,BCGG,CEK2,EKW1,EKW2,GK1} and as the universality associated with the break-up of invariant surfaces  \cite{Shen,McK1,McK2,ME},  and in Hamiltonian flows \cite{ED,AK,AKW,Koch1,Koch2,Koch3,GK,Gai1,Kocic}. 

To prove universality one usually introduces a {\it renormalization} operator on a functional space, and demonstrates that this operator has a {\it hyperbolic fixed point}.  The renormalization approach to universality has been very successful in one-dimensional dynamics, and has led to explanation of universality in unimodal maps \cite{Eps1,Eps2,Lyu}, critical circle maps \cite{dF1,dF2,Ya1,Ya2} and holomorphic maps with a  Siegel disk \cite{McM,Ya3,GaiYa}. There is, however, at present  no deep understanding of universality in conservative systems, other than in the ``trivial'' case of the universality for systems ``near integrability'' \cite{Koch1,Koch2,Gai1,Kocic,KLDM}. 

It is worth noting that universality in conservative systems seems to be completely different from that in one-dimensional and dissipative maps. As it has been shown in \cite{CEK1,dCLM,LM}, the case of very dissipative systems is largely reducible to the one-dimensional Feigenbaum-Coullet-Tresser universality.

For families of area-preserving maps a universal infinite period-doubling cascade was observed by several authors in the early 80's \cite{DP,Hel,BCGG,Bou,CEK2}.  The existence of a hyperbolic fixed point for the period-doubling renormalization operator has been proved with computer-assistance in \cite{EKW2}. 


In \cite{GJ} we used the method of covering relations (see, e.g. \cite{Z97,ZG04,KWZ07,Z09,CAPD}) in rigorous computations to construct hyperbolic sets for all maps in some neighborhood of the fixed point of the renormalization operator.  The Hausdorff dimension of these hyperbolic sets has been estimated with the help of the Duarte Distortion Theorem (see, e.g. \cite{Duarte1}) which enables one to use the distortion of a Cantor set to find bounds on its dimension. 

In this paper, we prove that {\it infinitely renormalizable} maps in a neighborhood of existence of the hyperbolic sets also admit  a ``stable'' set. This set is a bounded invariant set, such that the maximal Lyapunov exponent at any point of this set is zero.  Together with our result from \cite{GJ}, this demonstrates that for all reversible area-preserving infinitely renormalizable maps in some neighborhood of the renormalization fixed point, there are coexisting hyperbolic and stable sets.

We also address the issues of rigidity of the stable set and invariance of its Hausdorff dimension. Similar issues have been investigated in \cite{dCLM}  for attractors of very dissipative two-dimensional maps, where it has been shown that the regularity of conjugacy of attractors for two infinitely renormalizable maps $F$ and $\tilde{F}$ has a definite upper bound 
\begin{equation}\label{reg_bound}
\alpha \le {1 \over 2} \left(1 +\min\left\{ {\log {\rm Jac}(F)\over \log {\rm Jac}(\tilde{F}) }, {\log {\rm Jac}(\tilde{F})\over \log {\rm Jac}({F}) } \right\} \right),
\end{equation}
where $\rm Jac(F)$ is the ``average'' Jacobian of the map $F$. The authors of \cite{dCLM} put forward two questions: 1) whether the Hausdorff dimension of the attractor of an infinitely renormalizable map depends only on its average Jacobian, and 2) how regular is the conjugacy when ${\rm Jac}(F)={\rm Jac}(\tilde{F})$. In this regard, we obtain a partial result along similar lines in the ``extreme'' case of area-preserving maps (constant Jacobian equal to one): we prove that there exists a subset of locally infinitely renormalizable maps such the actions of any two maps from this subset on their stable sets are conjugate by a bi-Lipschitz map which preserves the Hausdorff dimension. 

We can not make a definite conclusion about whether this subset is equal to the whole set of locally infinitely renormalizable maps, or strictly smaller, because a sharp bound on the convergence rate of renormalizations of infinitely renormalizable maps is not known to date.

Finally, we provide an upper bound on the Hausdorff dimension of the stable set for all infinitely renormalizable maps.

\medskip

\section{Renormalization for area-preserving reversible maps} 

An ``area-preserving map'' will mean an exact symplectic diffeomorphism of a subset of ${\fR}^2$ onto its image.

Recall, that an area-preserving map can be uniquely specified by its generating function $S$:
\begin{equation}\label{gen_func}
\left( x \atop -S_1(x,y) \right) {{ \mbox{{\small \it  F}} \atop \mapsto} \atop \phantom{\mbox{\tiny .}}}  \left( y \atop S_2(x,y) \right), \quad S_i \equiv \partial_i S.
\end{equation}

Furthermore, we will assume that $F$ is reversible, that is 
\begin{equation}\label{reversible}
T \circ F \circ T=F^{-1}, \quad {\rm where} \quad T(x,u)=(x,-u).
\end{equation}

For such maps it follows from $(\ref{gen_func})$ that 
$$S_1(y,x)=S_2(x,y) \equiv s(x,y),$$
and
\begin{equation}\label{sdef}
\left({x  \atop  -s(y,x)} \right)  {{ \mbox{{\small \it  F}} \atop \mapsto} \atop \phantom{\mbox{\tiny .}}} \left({y \atop s(x,y) }\right).
\end{equation}

It is this ``little'' $s$ that will be referred to below as ``the generating function''. If the equation $-s(y,x)=u$ has a unique differentiable solution $y=y(x,u)$, then the derivative of such a map $F$ is given by the following formula:

\begin{equation}\label{Fder}
DF(x,u)=\left[ 
\begin{array}{c c}
-{s_2(y(x,u),x) \over s_1(y(x,u),x)} &  -{1 \over s_1(y(x,u),x)} \\
s_1(x,y(x,u))-s_2(x,y(x,u)) {s_2(y(x,u),x) \over s_1(y(x,u),x)}  & -{s_2(x,y(x,u)) \over s_1(y(x,u),x)} 
\end{array}
\right]. 
\end{equation}

The period-doubling phenomenon can be illustrated with the area-preserving H\' enon family (cf. \cite{Bou}) :
$$ H_a(x,u)=(-u +1 - a x^2, x).$$

Maps $H_a$ have a fixed point $((-1+\sqrt{1+a})/a,(-1+\sqrt{1+a})/a) $ which is stable for $-1 < a < 3$. When $a_1=3$ this fixed point becomes unstable, at the same time an orbit of period two is born with $H_a(x_\pm,x_\mp)=(x_\mp,x_\pm)$, $x_\pm= (1\pm \sqrt{a-3})/a$. This orbit, in turn, becomes unstable at $a_2=4$, giving birth to a period $4$ stable orbit. Generally, there  exists a sequence of parameter values $a_k$, at which the orbit of period $2^{k-1}$ turns unstable, while at the same time a stable orbit of period $2^k$ is born. The parameter values $a_k$ accumulate on some $a_\infty$. The crucial observation is that the accumulation rate
\begin{equation}
\lim_{k \rightarrow \infty}{a_k-a_{k-1} \over  a_{k+1}-a_k } = 8.721...
\end{equation} 
is universal for a large class of families, not necessarily H\'enon.

Furthermore, the $2^k$ periodic orbits scale asymptotically with two scaling parameters
\begin{equation}
\lambda=-0.249 \ldots,\quad \mu=0.061 \ldots
\end{equation}

To explain how orbits scale with $\lambda$ and $\mu$ we will follow \cite{Bou}. Consider an interval $(a_k,a_{k+1})$ of parameter values in a ``typical'' family $F_a$. For any value $\alpha \in (a_k,a_{k+1})$ the map $F_\alpha$ possesses a stable periodic orbit of period $2^{k}$. We fix some $\alpha_k$ within the interval $(a_k,a_{k+1})$ in some consistent way; for instance, by requiring that the restriction of $F^{2^{k}}_{\alpha_k}$ to a neighborhood of a stable periodic point in the $2^{k}$-periodic orbit is conjugate, via a diffeomorphism $H_k$, to a rotation with some fixed rotation number $r$.  Let $p'_k$ be some unstable periodic point in the $2^{k-1}$-periodic orbit, and let $p_k$ be the further of the two stable $2^{k}$-periodic points that bifurcated from $p'_k$.  Denote with $d_k=|p'_k-p_k|$, the distance between $p_k$ and $p'_k$. The new elliptic point $p_k$ is surrounded by invariant ellipses; let $c_k$ be the distance between $p_k$ and $p'_k$ in the direction of the minor semi-axis of an invariant ellipse surrounding $p_k$, see Figure \ref{bifGeom}. Then,
$${1 \over \lambda}=-\lim_{k \rightarrow \infty}{ d_k \over d_{k+1}},\quad
{\lambda \over \mu}=-\lim_{k \rightarrow \infty}{ \rho_k \over \rho_{k+1}}, \quad
{1 \over \lambda^2}=\lim_{k \rightarrow \infty}{ c_k \over c_{k+1}},
$$
where $\rho_k$ is the ratio of the smaller and larger eigenvalues of $D H_k(p_k)$. 

\begin{figure}[t]
\psfrag{p_k}{$p_k$}
\psfrag{p'_k}{$p'_k$}
\psfrag{c_k}{$c_k$}
\psfrag{d_k}{$d_k$}
\begin{center}
\includegraphics[width=0.7 \textwidth]{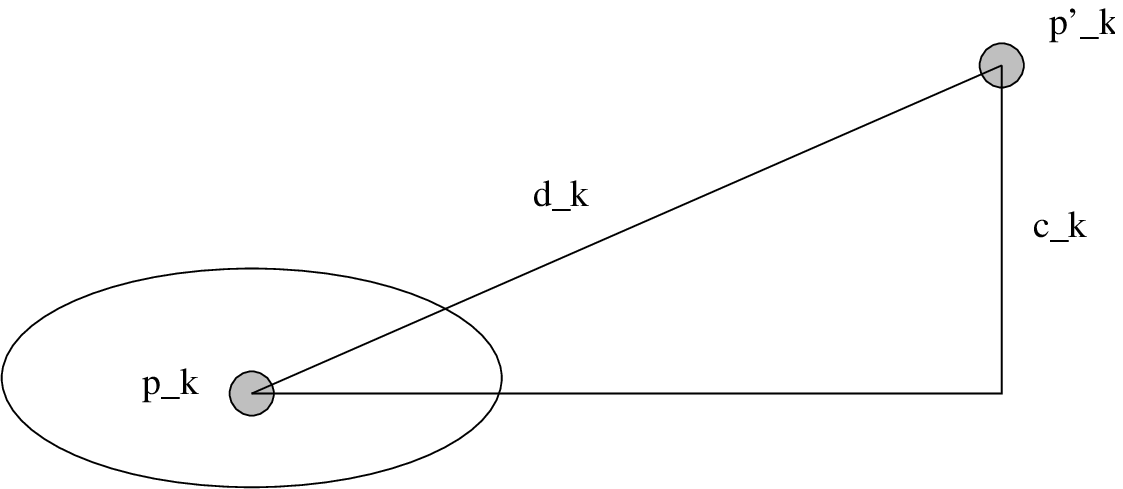}
\caption{The geometry of the period doubling. $p_k$ is the further elliptic point that has bifurcated from the hyperbolic point $p'_k$.}\label{bifGeom}
\end{center}
 \end{figure}

This universality can be explained rigorously if one shows that the {\it renormalization} operator
\begin{equation}\label{Ren}
R[F]=\Lambda^{-1}_F \circ F \circ F \circ \Lambda_F,
\end{equation}
where $\Lambda_F$ is some $F$-dependent coordinate transformation, has a fixed point, and the derivative of this operator is hyperbolic at this fixed point.

It has been argued in \cite{CEK2}  that $\Lambda_F$ is a diagonal linear transformation. Furthermore, such $\Lambda_F$ has been used in \cite{EKW1} and \cite{EKW2} in a computer assisted proof of existence of a reversible renormalization fixed point $F_*$ and hyperbolicity of the operator $R$. 

We will now derive an equation for the generating function of the renormalized map $\Lambda_F^{-1} \circ F \circ F \circ \Lambda_F$.

Applying a reversible $F$ twice we get
$$
 \left({x'  \atop  -s(z',x')} \right) {{ \mbox{{\small \it  F}} \atop \mapsto} \atop \phantom{\mbox{\tiny .}}} \left({z' \atop s(x',z')} \right) = \left({z'  \atop  -s(y',z')} \right) {{ \mbox{{\small \it  F}} \atop \mapsto} \atop \phantom{\mbox{\tiny .}}} \left({y'\atop  s(z',y')} \right).
$$

According to \cite{CEK2}  $\Lambda_F$ can be chosen to be a linear diagonal transformation:  

$$\Lambda_F(x,u)=(\lambda x, \mu u).$$

We, therefore, set  $(x',y')=(\lambda x,  \lambda y)$, $z'(\lambda x, \lambda y)= z(x,y)$ to obtain:

\begin{equation}\label{doubling}
\left(\!{x  \atop  -{ 1 \over \mu } s(z,\lambda x)} \!\right) \!{{ \mbox{{\small $\Lambda_F$}} \atop \mapsto} \atop \phantom{\mbox{\tiny .}}} \!\left(\!{\lambda x  \atop  -s(z,\lambda x)} \!\right) \!{{ \mbox{{\small \it  F $ \circ$ F}} \atop \mapsto} \atop \phantom{\mbox{\tiny .}}}\!\left(\!{\lambda y \atop s(z,\lambda y)}\! \right)   {{ \mbox{{\small \it  $\Lambda_F^{-1}$}} \atop \mapsto} \atop \phantom{\mbox{\tiny .}}} \left(\!{y \atop {1 \over \mu } s(z,\lambda y) }\!\right),
\end{equation}
where $z(x,y)$ solves
\begin{equation}\label{midpoint}
s(\lambda x, z(x,y))+s(\lambda y, z(x,y))=0.
\end{equation}

If the solution of $(\ref{midpoint})$ is unique, then $z(x,y)=z(y,x)$, and it follows from $(\ref{doubling})$ that the generating function of the renormalized $F$ is given by 
\begin{equation}
\tilde{s}(x,y)=\mu^{-1} s(z(x,y),\lambda y).
\end{equation}

One can fix a set of normalization conditions for $\tilde{s}$ and $z$ which serve to determine scalings $\lambda$ and $\mu$ as functions of $s$. For example, the normalization $s(1,0)=0$ is reproduced for $\tilde{s}$ as long as $z(1,0)=z(0,1)=1.$ In particular, this implies that $s(\lambda, 1)+s(0, 1)=0.$ Furthermore, the condition $\partial_1 s(1,0)=1$ is reproduced as long as $\mu=\partial_1 z (1,0).$

We will now summarize the above discussion in the following definition of the renormalization operator acting on generating functions originally due  to the authors of \cite{EKW1} and \cite{EKW2}:

\begin{definition}
\phantom{a}

\begin{eqnarray}\label{ren_eq}
\nonumber \\ {\cR}_{EKW}[s](x,y)=\mu^{-1} s(z(x,y),\lambda y),
\end{eqnarray}
where
\begin{eqnarray}
\label{midpoint_eq} 0&=&s(\lambda x, z(x,y))+s(\lambda y, z(x,y)), \\
0&=&s(\lambda,1)+s(0,1) \quad {\rm and} \quad \mu=\partial_1 z (1,0).
\end{eqnarray}
\end{definition}

\medskip

\begin{definition}
The Banach space of functions  $s(x,y)=\sum_{i,j=0}^{\infty}c_{i j} x^i y^j$, analytic on a bi-disk
$$ |x-0.5|<\rho, |y-0.5|<\rho,$$
for which the norm
$$\|s\|_\rho=\sum_{i,j=0}^{\infty}|c_{i j}|\rho^{i+j}$$
is finite, will be referred to as $\cA(\rho)$.

$\cA_s(\rho)$ will denote its symmetric subspace $\{s\in\cA(\rho) : s_1(x,y)=s_1(y.x)\}$.
\end{definition}

\medskip

As we have already mentioned, the following has been proved with the help of a computer in \cite{EKW1} and \cite{EKW2}:
\begin{oldthm}\label{EKWTheorem}
There exist a polynomial $s_{\rm a} \in \cA_s(\rho)$ and  a ball $\cB_r(s_{\rm a}) \subset \cA_s(\rho)$, $r=6.0 \times 10^{-7}$, $\rho=1.6$, such that the operator ${\cR}_{EKW}$ is well-defined and analytic on $\cB_r(s_{\rm a})$. 

Furthermore, its derivative $D {\cR}_{EKW} \arrowvert_{\cB_r(s_{\rm a})}$ is a compact linear operator, and has exactly two eigenvalues $\delta_1$ and $\delta_2$ of modulus larger than $1$, while 
$${\rm spec}(D {\cR}_{EKW} \arrowvert_{\cB_r(s_{\rm a})}) \setminus \{\delta_1,\delta_2 \} \subset \{z \in \C: |z| \le \nu \},$$ 
where 
\begin{equation}\label{contr_rate}
\hspace{3.0cm}\nu <  0.85.
\end{equation}

Finally, there is an $s^* \in \cB_r(s_{\rm a})$ such that
$$\cR_{EKW}[s^*]=s^*.$$
The scalings $\lambda^*$ and $\mu^*$ corresponding to the fixed point $s^*$ satisfy
\begin{eqnarray}
\label{lambda} \lambda_* \in [-0.24887681,-0.24887376], \\
\label{mu} \mu_* \in [0.061107811, 0.061112465].
\end{eqnarray}
 \end{oldthm}

\medskip
 \begin{remark}
The bound $(\ref{contr_rate})$ is not sharp. In fact, a (lower) bound on the largest eigenvalue of $D {\cR}_{EKW}(s_*)$, restricted to the tangent space of the stable manifold, is not known.
 \end{remark}
\medskip

The interval enclosures of $\lambda_*$ and $\mu_*$ will be denoted 
\begin{eqnarray}
\label{blambda} \blambda &\equiv& [\lambda_-,\lambda_+]; \quad \lambda_-=-0.24887681, \quad \lambda_+=-0.24887376,\\
\label{bmu} \bmu & \equiv& [\mu_-,\mu_+]; \quad \mu_-=0.061107811, \quad \mu_+=0.061112465.
\end{eqnarray}

The corresponding interval enclosure for the linear map 
$$\Lambda_*\equiv \left[\lambda_* \ 0 \atop 0 \ \mu_* \right]$$ 
will be denoted $\bLambda$; if $(x,u) \in \fC^2$, then
\begin{equation}
 \bLambda(x,u) \equiv \left\{(\lambda x,\mu u) \in \fC^2: \lambda \in \blambda, \mu \in \bmu  \right\}. 
\end{equation}

The bound on the fixed point generating function $s^*$ will be called $\bs$:
\begin{equation}\label{bs}
\bs \equiv \left\{ s \in \cA(\rho) : \|s-s_{\rm a}\|_{\rho} \le  r=6.0 \times 10^{-7}  \right\},
\end{equation}
while the bound on the renormalization fixed point $F_*$ will be referred to as  $\bF$:
\begin{equation}\label{bF}
\bF \equiv \left\{ F: (x,-s(y,x)) \mapsto (y,s(x,y)) : s \in \bs \right\},
\end{equation}
the third iterate of this bound will be referred to as $\bG$.

It follows from  Theorem $\ref{EKWTheorem}$, that there exists a codimension $2$ local stable manifold $W^s_\loc(s^*)\subset \cB_r(s_{\rm a})$.

\begin{definition}\label{Wdef}
A reversible map $F$ of the form (\ref{sdef}) such that $s\in W^s_\loc(s^*)$ is called infinitely renormalizable. The set of all reversible infinitely renormalizable maps is denoted by $\bW$.
\end{definition}

\begin{definition}\label{Wdefr}
The set of reversible maps $F$ of the form (\ref{sdef})  with $s \in \cB_\varrho(s^*)$  will be referred to as $\bF(\varrho)$. Denote,
$$\bW(\varrho) \equiv \bW \cap \bF(\varrho).$$
\end{definition}

Naturally, $\bW(\varrho)$ is invariant under renormalization if $\varrho$ is sufficiently small.

Compactness of $D {\cR}_{EKW} \arrowvert_{\cB_r(s_{\rm a}) }$ implies that for any $\omega \le \nu$ there exists a submanifold $\bW_\omega \subset \bW$ (of finite codimension in $\bW$) such that $\|R^k[F]-F_*\|_\rho \le {\rm const} \, \omega^k$ for all $F \in \bW_\omega$.

\begin{definition}\label{Wdefo}
Define
$$\bW_\omega(\varrho)=\bW(\varrho) \cap \bW_\omega.$$
\end{definition}

\medskip

\section{Hyperbolic sets for maps in $\bF$}

In this Section we will recall  some of our results from the satellite paper \cite{GJ}.

We will start by introducing several classical definitions which will be helpful in understanding our Theorem $\ref{PIMTHM2}$ below.

\begin{definition} \label{hyperbolicset}
Let $\cM$ be a smooth manifold, and let $F$ be a diffeomorphism of an open subset $\cU \subset \cM$ onto its image.

An invariant set $\cC$ is called \underline{hyperbolic} for the map $F$ if there is a Riemannian metric on a neighborhood $\cU$ of $\cC$, and $\beta<1<\delta$, such that for any $p \in \cC$ and $n \in \fN$ the tangent space $T_{F^n(p)}\cU$ admits a decomposition in two equivariant subspaces:
$$T_{F^n(p)}\cM  =  E^+_n \oplus E^-_n, \quad 
DF(F^n(p)) E^\pm_n = E^\pm_{n+1},$$
on which the sequence of differentials is hyperbolic:
$$\|DF(F^n(p))\arrowvert_{E^-_n}\| < \beta, \quad \| DF^{-1} (F^n(p)) \arrowvert_{E^+_{n+1}} \| <  \delta^{-1}.$$

The hyperbolic set $\cC$ is called \underline{locally maximal},  if there is a neighborhood $\cV$ of $\cC$ such that  $\cC=\cap_{n \in \fZ} F^n(\bar{\cV})$.
\end{definition}

\medskip

\begin{definition} \label{shift}
Let  $\{0,1,\ldots,N-1\}^{\fZ}$ be the space of all two-sided sequences of $N$ symbols:
$$\{0,1,\ldots,N-1\}^{\fZ} = \{\omega=(\ldots,\omega_{-1},\omega_0,\omega_1,\ldots): \omega_i=\{0,1,\ldots,N-1\}, i\in\fZ   \},$$
Define the \underline{Bernoulli shift} on  $\{0,1,\ldots,N-1\}^{\fZ}$ as
$$\sigma_N (\omega)=\omega', \quad \omega'_n=\omega_{n+1}.$$
\end{definition}

\medskip

\begin{definition} \label{HDim}
Let $\cX$ be a metric space, $\cA \subset \cX$ and $d \in [0,\infty)$. Suppose that $\cB=\{ B_i \}$ is some cover of $\cA$ whose elements are open sets. We will denote
\begin{equation}\label{Cd}
C_d[\cB] \equiv \sum_i {\rm diam}(B_i)^d.
\end{equation} 

The $d$-dimensional \underline{Hausdorff content} of $\cA$ is defined as 
\begin{equation}\label{Hcont}
C_d^H[\cA]={\rm inf} \left\{C_d[\cB]: \cB \, {\rm is \, a \, cover \, of \, } \cA \right\}.
\end{equation}

The \underline{Hausdorff dimension} of $\cA$ is defined as 
\begin{equation}
{\rm dim}_H(\cA) = {\rm inf} \left\{d \ge 0: C_d^H[\cA]=0\right\}.
\end{equation}
\end{definition}

\medskip

In \cite{GJ} we have demonstrated that all maps in a neighborhood of the fixed point admit a hyperbolic set in their domain of analyticity. 

\begin{oldthm}\label{PIMTHM2}
 The following holds for all $F \in \bF$.

\medskip

\noindent i)  There exist connected open sets $\cD \subset \fC^2$ and $\cD_3 \subset \fC^2$ such that the maps $F$ and  $G \equiv F \circ F \circ F$ are analytic on $\cD$ and $\cD_3$, respectively.  

\medskip

\noindent ii)
The map $F$ possesses a hyperbolic fixed point $p_0=p_0(F) \in \cD$, such that
\begin{itemize}
\item[1)] $\cP_x p_0 \in ( 0.57761843, 0.57761989)$, and $\cP_u p_0=0$, where $\cP_{x,u}$ are projections on the $x$ and $u$ coordinates;
\item[2)] $D F(p_0)$ has two negative eigenvalues. 
\begin{eqnarray}
\nonumber e_+ &\in& (-2.05763559,-2.05759928),\\
\nonumber e_- &\in& (-0.48601715,-0.48598084).
\end{eqnarray} 
\end{itemize}

\medskip

\noindent iii) The map $G$ admits a locally maximal invariant hyperbolic set $\cC_G$: 
$$ \cC_G=\bigcap_{n \in \fZ} G^{-n}(\Delta),$$
and
$$G \arrowvert_{\cC_G} \, { \approx \atop  \mbox{{\small \it  homeo}}   }  \, \sigma_2 \arrowvert_{\{0,1\}^{\fZ}},$$ 
where $\Delta=\Delta_0 \cup \Delta_1$ and $\Delta_0 \subset \cD_3$, $\Delta_1 \subset \cD_3$ are compact sets, diffeomorphic to rectangles, with non-empty interior, that constitute a Markov partition for $G \arrowvert_{\cC_G}$.

Furthermore, the Hausdorff dimension of $\cC_G$ satisfies:
$$0.76594 > {\rm dim}_H(\cC_G) > \varepsilon,$$
where $\varepsilon \approx 0.00013 \, e^{-7499}$ is strictly positive.  

\medskip

\noindent iv) The local stable manifold $\cW_{\rm loc}^s(p_0)  \cap \Delta_0$ is a graph over the $x$-axis with the angle of the slope bounded away from $0$ and $\pi/2$.

\end{oldthm}

\medskip

\begin{remark}
The bounds on the rectangles $\Delta_0$ and $\Delta_1$ of the Markov  partition for $\cC_G$ are given in Table $\ref{bigHorseshoe}$.  
\end{remark}

\medskip

One can  construct a convergent sequence of approximations of the hyperbolic  sets $\cC_G$ in a straightforward way. Define recursively:
\begin{equation}\label{UkG}
 \cU_G^1 \equiv G(\Delta) \cap G^{-1}(\Delta) \quad {\rm and} \quad \cU_G^k \equiv   G(\cU_G^{k-1}) \cap G^{-1} (\cU_G^{k-1}).
\end{equation}  

Each of the sets $\cU_G^k$ contains $2 \cdot 4^k$ components 
$\cU_G^{k,n}$, $n=1..2 \cdot 4^k$. The following Lemma has been proved in \cite{GJ}.

\begin{oldlemma}\label{lMSBds}
Let 
$$\rho_{k,n}={\rm sup}_{B_\rho \subset \cU_G^{k,n}} (\rho),$$
and set $\rho_k=\min_n \{ \rho_{k,n} \}$. There exist constants $C>0$ and $c>0$ such that
$${\rm diam}\left(\cU_G^{k,n}\right) \le C \kappa_+^k, \quad  \rho_k \ge  c \, \kappa_-^k,$$
where $\kappa_-=0.0371$ and $\kappa_+=0.1642$.
\end{oldlemma}

\medskip

In this paper we will complement Theorem $\ref{PIMTHM2}$, and show that the third iterate of $F$ also supports a stable set in its domain of analyticity.

\begin{table}[t]
\begin{center}
\begin{footnotesize}
\begin{tabular}{c|ccc}
Component & Centre & ``Stable'' Scale & ``Unstable'' Scale\\ \hline
$\Delta_0'$ & $(0.670198,  0.0)$ & $0.083$ & $0.083$ \\
$\Delta_1'$ & $(-0.441811, 0.0)$ & $0.0655$ & $0.0655$
\end{tabular}
\begin{center}
\caption{The rectangles that approximate the Markov partition for the horseshoe of $G$. The rectangles are spanned by vectors  ${\bf e}^s_0=(0.788578889012330, -0.614933602760558)$,  ${\bf e}^u_0=T({\bf e}^s_0)$  and ${\bf e}^s_1=(0.750925931392967773, 0.660386436536671957)$,  ${\bf e}^u_1=T({\bf e}^s_1)$, respectively.  The length of the sides of the rectangles $\Delta_0'$ and $\Delta_1'$ is $2 \cdot {\rm stable/unstable \quad \!\! scale} \cdot |{\bf e}^{u,s}_{0,1}|$.}\label{bigHorseshoe}
\end{center}
\end{footnotesize}
\end{center}
\end{table}

\medskip

\section{Statement of results}

Recall, that a map $\cH: \cX \rightarrow \cY$ between two metric spaces $\cX$ and $\cY$ is called {\it bi-Lipschitz}, if there is a constant $C \ge 1$, such that for any two points $p$ and $q$ in $\cX$
\begin{equation}\label{bi-Lip}
{1 \over C} \, {\rm dist}_{\cX}(p,q) \le {\rm dist}_{\cY}(\cH(p),\cH(q)) \le C \, {\rm dist}_{\cX}(p,q),
\end{equation} 
i.e. if ${\rm dist}_{\cX}(p,q)$ and ${\rm dist}_{\cY}(\cH(p),\cH(q))$ are {\it commensurate}. We denote commensurability of two quantities by ``$\asymp$''.

A classical result from analysis states that such maps preserve the Hausdorff dimension.

Consider the dyadic group,
\begin{equation}\label{dyadic_group}
\{0,1\}^\infty=\underleftarrow{\lim}\{0,1\}^n,
\end{equation}
where $\underleftarrow{\lim}$ stands for the inverse limit. An element $w$ of the dyadic group can be represented as a formal power series $w \rightarrow \sum_{k=0}^\infty w_{k+1} 2^k$. The {\it odometer}, or the {\it adding machine}, $p: \{0,1\}^\infty \rightarrow  \{0,1\}^\infty$ is the operation of adding $1$ in this group. 

We are now ready to state our main theorem. 

\begin{mainthm}\label{MTHM2}
There exists $\varrho>0$ such that any $F \in \bW(\varrho)$   admits a ``stable'' Cantor set $\cC_F^{\infty} \subset \cD$,  that is the set on which the  maximal Lyapunov exponent is equal to zero, with the following properties.

\begin{itemize}

\item[1)] The set $\cC^\infty_F$ is a Hausdorff limit of invariant hyperbolic sets of vanishing hyperbolicity.

\item[2)] The Hausdorff dimension of $\cC^\infty_F$ satisfies
$${\rm dim}_H(\cC_F^{\infty}) \le 0.5324.$$

\item[3)] The restriction of the dynamics $F \arrowvert_{\cC^\infty_F}$ is topologically conjugate to the adding machine.

\item[4)]  If
$$\omega \le \min\left\{{\mu_* \over |\lambda_*|},{\inf_{p \in \cD, v \in \field{R}^2} \Arrowvert \Lambda_* \cdot D F_*(p) \cdot v \Arrowvert \over \sup_{p \in \cD, v \in \field{R}^2} \Arrowvert \Lambda_* \cdot D F_*(p) \cdot v \Arrowvert }  \right\},$$
then for all  $F$ and $\tilde{F}$  in $\bW_\omega(\varrho)$
$${\rm dim}_H(\cC_F^{\infty}) = {\rm dim}_H(\cC_{\tilde{F}}^{\infty}),$$
and 
$$F \arrowvert_{\cC_F^{\infty}}  \, { \approx \atop { \cH} }    \, \tilde{F} \arrowvert_{\cC_{\tilde{F}}^{\infty}},$$
where $\cH$ is  a bi-Lipschitz map.
\end{itemize}
 
\end{mainthm}
 
\begin{remark}
We have obtained the following rigorous computer bound
$${\inf_{p \in \cD, v \in \field{R}^2} \Arrowvert \Lambda_* \cdot D F_*(p) \cdot v \Arrowvert \over \sup_{p \in \cD, v \in \field{R}^2} \Arrowvert \Lambda_* \cdot D F_*(p) \cdot v \Arrowvert } \le 0.0581,$$
which is clearly dependent on our choice of coordinates for $F$'s. At the same time,  $\omega \le \mu_* / |\lambda_*| <0.246$ (independently of the choice of a coordinate system in which maps $F$ are considered), while $\nu$, the renormalization convergence rate on $\bW(\varrho)$, is less than $0.85$. Therefore, it seems likely that the submanifold  $\bW_\omega(\varrho) \ne  \bW(\varrho)$.
\end{remark}


\medskip

\section{Some notation and definitions}

We will use the following notation for the sup norm of a function $h$ and a transformation $H$ defined on 
some set $\cS \subset \fR^2 \! \quad {\rm or} \quad \! \fC^2$:
\begin{eqnarray}
\arrowvert h \arrowvert_\cS \equiv \sup_{(x,u) \in \cS }\{ |h| \},\\
\arrowvert H \arrowvert_\cS \equiv \max \{\sup_{(x,u) \in \cS }\{|\cP_x H| \}, \sup_{(x,u) \in \cS}\{|\cP_u H| \} \},
\end{eqnarray}
where $\cP_x$ and $\cP_u$ are projections on the corresponding components. We will also use the notation $| \cdot |$ for the $l_2$ norm for vectors in $\fR^2$.


 With $D\bG: p \mapsto D\bG(p)$ we denote an interval matrix valued function such that 
$$
[DG(p)]_{ij} \in [D\bG(p)]_{ij}, \quad \textrm{for all } G\in \bG,\,p\in\cD_3,
$$
where $\cD_3$ is the domain of $\bG$, and the bound on the operator norm of $DG$ for $G=F \circ F \circ F$, $F\in \bF$ on a set $\cS$ will be denoted
$$\| D\bG \|_\cS \equiv \sup_{F \in \bF }\left\{ \| D (F \circ F \circ F) \|_{\cS}\right\}.$$


We will also use the following abbreviations for maps, transformations and scalings
\begin{eqnarray}
G_j &\equiv & R^j[G],\\
\Lambda_{k,G} & \equiv & \Lambda_G \circ \Lambda_{R[G]} \circ \ldots \Lambda_{R^{k-1}[G]}= \Lambda_{G_0} \circ \Lambda_{G_1} \circ \ldots \circ \Lambda_{G_{k-1}},\\
\lambda_{k,G} &\equiv & \lambda_{G_0}  \lambda_{G_1}  \ldots \lambda_{G_{k-1}},\\
\mu_{k,G} &\equiv & \mu_{G_0}  \mu_{G_1}  \ldots \mu_{G_{k-1}}.
\end{eqnarray}

\medskip

\section{A stable invariant set as a Hausdorff limit of hyperbolic sets}

According to Theorem $\ref{PIMTHM2}$ the fixed point $F_*$ possesses a hyperbolic set for its third iterate. By the stability property of such sets, there exists a neighborhood 
$\cB_{r'}(s^*)$ such that all maps $F$ of the form $(\ref{sdef})$ with $s \in \cB_{r'}(s^*)$ also have a hyperbolic set  $\cC_G$ for $G \equiv F \circ F \circ F$, and the action of $G$ on $\cC_G$ is topologically conjugate to that of $G_* \equiv F_* \circ F_* \circ F_*$ on $\cC_*$: 
$$G \circ H_G = H_G \circ G_* \arrowvert_{\cC_* }. $$

In what follows, we consider maps $F \in \bW(\varrho)$ (see Definition $\ref{Wdefr}$) where 
$$\varrho \le  \min\{ r',r\}.$$
The following holds on $\Lambda^k_*(\cC_*) $ for all $F \in \bW(\varrho)$:
\begin{eqnarray}
 \nonumber \Lambda_{k,G} \circ  H_{G_k} \circ \Lambda^{-k}_* &= & \Lambda_{k,G} \circ  G^{-1}_k \circ 
H_{G_k}  \circ G_* \circ \Lambda^{-k}_* \\
\nonumber &= & \Lambda_{k,G} \circ  G^{-1}_k \circ  \Lambda_{k,G}^{-1} \circ  \Lambda_{k,G}  \circ
H_{G_k} \circ \Lambda^{-k}_* \circ \Lambda_*^k  \circ G_* \circ \Lambda^{-k}_* \\
\nonumber &= & G^{-2^k} \circ  \Lambda_{k,G}  \circ H_{G_k} \circ \Lambda^{-k}_* \circ G_*^{2^k}.
\end{eqnarray}
Therefore, the transformation
$$\HkG =\Lambda_{k,G} \circ H_{G_k} \circ \Lambda_*^{-k}$$
is a topological conjugacy of iterates $G^{2^k}$ and $G_*^{2^k}$ on  $\Lambda^k_*(\cC_*)$:

Define
$$\iHkG  \equiv G^{i}\circ \HkG\circ G^{-i}_*, \quad 0 \le i \le 2^k-1.$$

Clearly,
\begin{equation}\label{Gconj}
G^{2^k} \circ \!\!\iHkG = \!\! \iHkG \circ G^{2^k}_* \quad {\rm and} \quad 
G \circ \!\!\iHkG=\!\!\! \phantom{a}_{i+1}\HkG \circ G_*.
\end{equation}
on $\cC_*^{k,i} \equiv \cC_{G_*}^{k,i}$, where
\begin{eqnarray}
\nonumber \cC_G^{k,i} &\equiv& G^i (\HkG (\Lambda_*^k(\cC_*))) \\
\nonumber &=&G^i (\Lambda_{k,G} ( H_{G_k}(\cC_*))) \\
\nonumber &=& G^i (\Lambda_{k,G} ( \cC_{G_k}))\\
 &=& \!\!\iHkG (G_*^{i}(\Lambda_*^k(\cC_*))).
\end{eqnarray}

Also, define the following sequence of sets

\begin{equation} \label{Gk}
\cC^k_G \equiv  \bigcup _{i=0}^{2^k-1} \cC_G^{k,i},
\end{equation}
and their covers  $\cV_G^k \supset \cC_G^k$:  
$$ \cV_G^k\equiv   \bigcup_{i=0}^{2^k-1} \cV_G^{k,i},$$
where 
$$\cV_G^{k,0} \equiv\Lambda_{k,G}(\cU_{G_k}^k),$$
$\cU^k_G$ is as in $(\ref{UkG})$, and
$$ \cV_G^{k,i} \equiv G^i (\Lambda_{k,G}(\cU_{G_k}^k)) \cap G^{i-2^k} (\Lambda_{k,G}(\cU_{G_k}^k)).
$$
Clearly, the map $\cH_{k,G}$ defined as 
$$\cH_{k,G}\arrowvert_{\cC_*^{k,i}} \equiv \iHkG,$$
is a conjugacy of $G$ and $G_*$ on $\cC^k_*$:
\begin{equation}\label{Hconj}
G \circ \cH_{k,G} \arrowvert_{\cC^k_*} = \cH_{k,G} \circ G_* \arrowvert_{\cC^k_*}.
\end{equation}
\medskip

In the rest of this Section we will be studying the sequence of sets $\cC_G^k$. We will demonstrate that the limit set $\cC_G^\infty$  exists, is {\it stable}, in the sense that the maximal Lyapunov exponent on $\cC_G^\infty$ is zero,  bounded, closed and {\it invariant} under $G$. We will also show that there exists an $\omega>0$ such that for any $F  \in \bW_\omega(\varrho)$ the sets $\cC^\infty_G$ are {\it weakly rigid}: there exists a bi-Lipschitz (see $(\ref{bi-Lip})$) conjugacy $\cH_G$ between $\cC_*^\infty$ and  $\cC_G^\infty$.

The stable set $\cC_G^\infty$ is an analogue of the attracting set for dissipative H\'enon-like maps constructed \cite{dCLM}. The (more standard) approach of \cite{dCLM} is based on the so called presentation functions; it also demonstrates that the attracting set is Cantor and that the restriction of the dynamics to it is homeomorphic to the adding machine. We outline a similar procedure in Section \ref{Cantor_Set}.

The approach of this Section is more technical than that method of presentation functions; one of the our goals in pursuing it was to demonstrate that {\it the stable set is a Hausdorff limit of hyperbolic Cantor sets with vanishing hyperbolicity}. On the other hand the more compact method of presentation functions shows that {\it the stable set is indeed Cantor, and that the stable dynamics is that of an odometer (an adding machine)}.

We will first demonstrate boundedness of sets $\cC^k_G$.

Given a set $\cS \subset \cD_3$ on which an iterate $G^i$ for all $G \in \bG$ is defined, we use the notation $\bG^i(\cS)$ as a shorthand for $ \cup_{G\in\bG} G^i(\cS)$.

Notations $\bLambda^n(\cS)$ and $\bT_n(\cS)$ are used in a similar sense.

\medskip

\begin{lemma}\label{4-set}
 For all $F \in \bW(\varrho)$,  the sets $\cC^k_G$ are bounded, in particular,  $\cC^k_G \subset \cE$  for all $k \ge 1$, where  
$$\cE \equiv  \bigcup_{i=1}^4 \cE_i,$$
 and 
\begin{eqnarray}
\nonumber \cE_1&=&\{(x,u) \in \fR^2: {(x+0.0328)^2 \over 0.169^2}+{u^2 \over 0.010683153^2}< 1\},\\ 
\nonumber \cE_2 &\equiv&  \bG(\cE_1), \quad 
\cE_3 \equiv  \bLambda(\cE_2 \cup \cE_4),\quad
\cE_4 \equiv T(\bG(\cE_1)).
\end{eqnarray}
\end{lemma}

\begin{figure}[t]
\begin{center}
\begin{tabular}{c c}
\resizebox{60mm}{!}{\includegraphics[angle=-90]{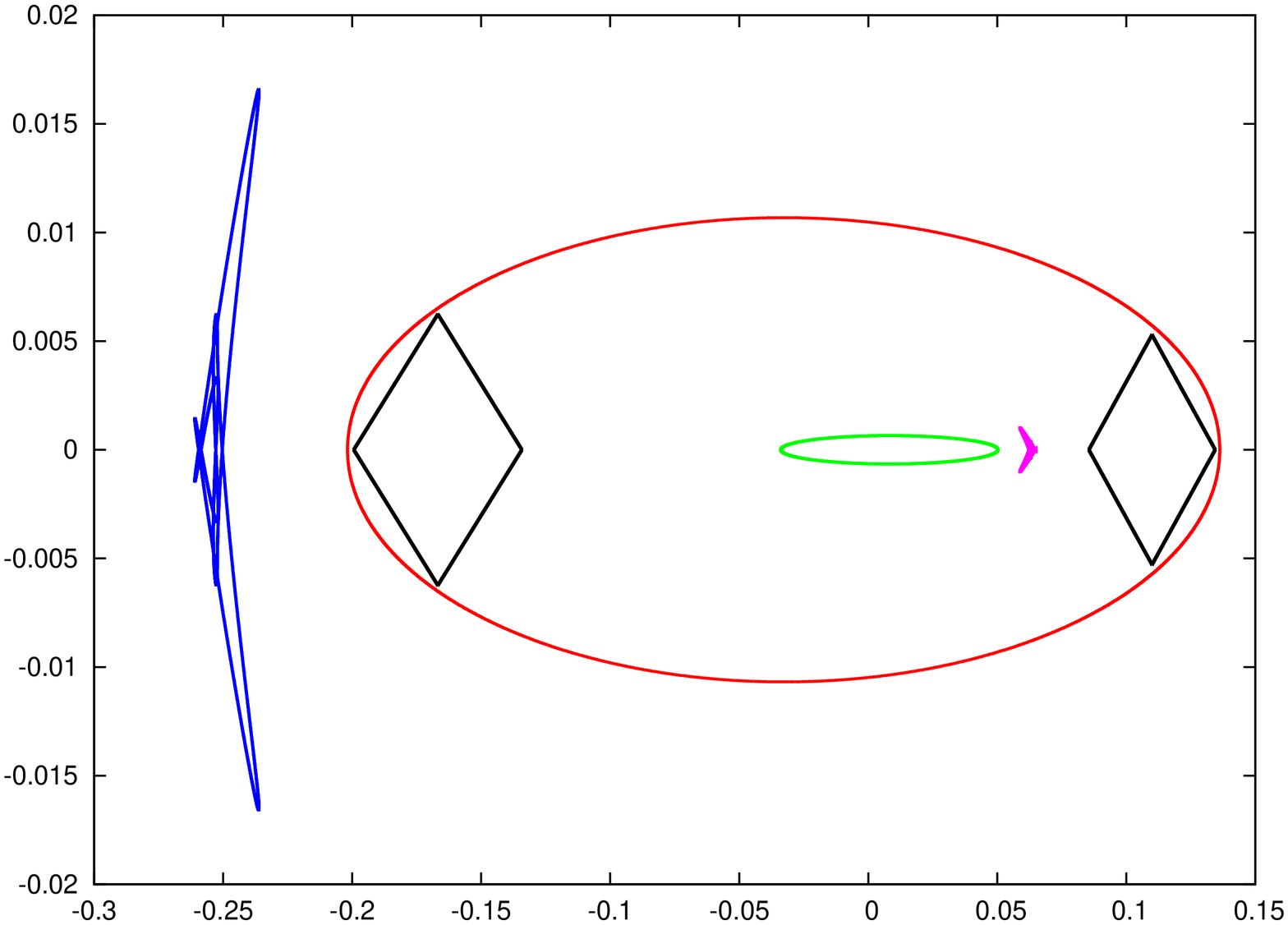}}&  \resizebox{60mm}{!}{\includegraphics[angle=-90]{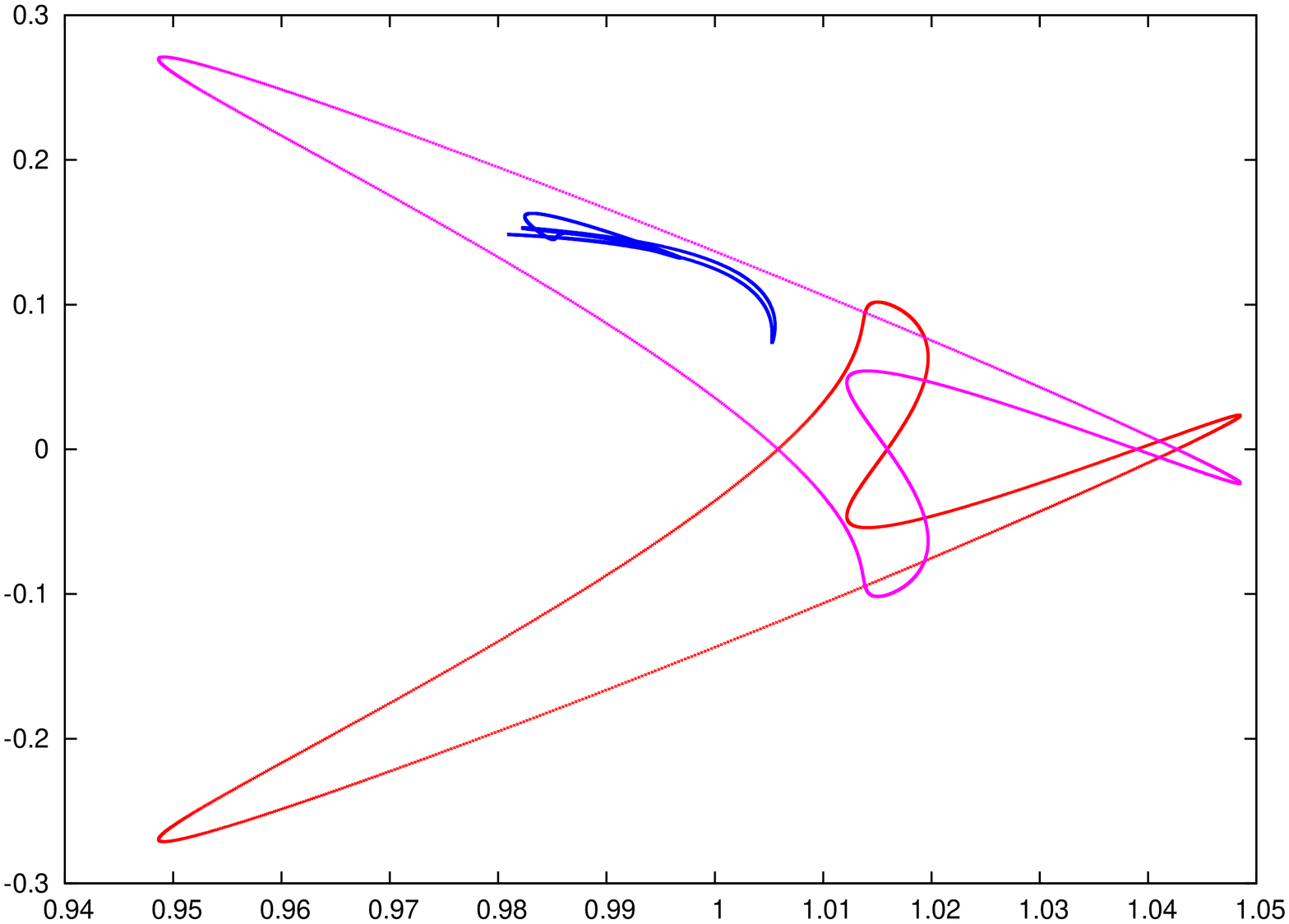}} \\
a) & b) 
\end{tabular} 
\caption{a) Sets $\cE_1$ (red), $\cE_3$ (blue), $\bLambda(\cE_1)$ (green) and $\bLambda(\cE_3)$ (magenta).  The two components of the Markov partition, rescaled by $\bLambda$,   are drawn in black.  b) Sets $\cE_2$ (red), $\cE_4$ (magenta) and $\bG(\cE_3)$ (blue).}\label{BoundedFig}
\end{center}
\end{figure}


\begin{proof}
Let $i<2^k$ for some $k \in \fN$. We write $i$ in its binary representation: 
$$i=\alpha_0 2^0 +\alpha_1 2^1+\ldots +\alpha_{k-1} 2^{k-1}, \quad \alpha_j=0,1.$$
Let $\{j_{l} \}_{l=1}^m$,  be the index set such that $\alpha_{j_l } \ne 0$:
\begin{equation}\label{binaryi}
i=2^{j_1} +2^{j_2}+\ldots +2^{j_m}, \quad j_m \le k-1, \quad m\le k.
\end{equation}

 Consider $G^i$ on a subset of  $\cD_3$ where this iterate is defined:
\begin{eqnarray}\label{Gi}
\nonumber G^i&=&G^{2^{j_1}} \circ G^{2^{j_2}} \circ \ldots \circ G^{2 ^{j_m}} \\
\nonumber &=&\Lambda_{j_1,G} \circ \!\left[ \Lambda_{j_1,G}^{-1} \circ G^{2^{j_1}} \! \circ \Lambda_{j_1,G} \right] \circ \Lambda_{j_1,G}^{-1} \circ \Lambda_{j_2,G} \\ 
\nonumber &\phantom{=}&\phantom{\Lambda_{j_1,G}}\circ \left[ \Lambda_{j_2,G}^{-1} \circ  G^{2^{j_2}}\! \circ \Lambda_{j_2,G} \right] \circ \Lambda_{j_2,G}^{-1} \circ \ldots \circ \Lambda_{j_m,G} \\
\nonumber &\phantom{=}&\phantom{\Lambda_{j_1,G}}\circ  \left[ \Lambda_{j_m,G}^{-1} \circ G^{2 ^{j_m}} \!\circ \Lambda_{j_m,G} \right] \circ \Lambda_{j_m,G}^{-1} \\
\nonumber &=&\Lambda_{j_1,G} \circ  G_{j_1} \circ \Lambda_{j_1,G}^{-1}  \circ \Lambda_{j_2,G} \circ  G_{j_2}  \circ \ldots \circ \Lambda_{j_{m-1},G}^{-1} \circ \Lambda_{j_m,G} \circ G_{j_m} \circ \Lambda_{j_m,G}^{-1}.
\end{eqnarray}

For convenience, we will denote
$$T_{n,m,G}=\Lambda_{m,G}^{-1} \circ \Lambda_{n,G} \circ G_n, \quad T_{n,0,G}=\Lambda_{n,G} \circ G_n, \quad T_n=\Lambda_*^n \circ G_*, \quad \bT_n \equiv \bLambda^n \circ \bG,
$$
and also use the following notation for compositions of these maps:
\begin{eqnarray}
\label{Ttransform2}\TisG{i}{q}{l}&=&T_{j_q,j_{q-1},G} \circ T_{j_{q+1},j_q,G} \circ \ldots \circ  T_{j_l,j_{l-1},G},\\
\label{Ttransform3}\Tis{i}{q}{l}&=&T_{j_q-j_{q-1}} \circ T_{j_{q+1}-j_q} \circ \ldots \circ  T_{j_l-j_{l-1}}, \\
\label{Ttransform4}\bTis{i}{q}{l}&=&\bT_{j_q-j_{q-1}} \circ \bT_{j_{q+1}-j_q} \circ \ldots \circ  \bT_{j_l-j_{l-1}}, \\
\label{Ttransform5}\TiG{i}&=&\TisG{i}{1}{m}, \quad  \Ti{i}=\Tis{i}{1}{m}, \quad  \bTi{i}=\bTis{i}{1}{m}, 
\end{eqnarray}
where $j_0 \equiv 0$.
In this notation, the iterate $G^i$ can be written as 
\begin{equation}\label{GiT}
G^i= \TiG{i} \circ \Lambda_{j_m,G}^{-1}.
\end{equation}

We apply the formula $(\ref{GiT})$ to write the action of $G^i$ on $\HkG (\Lambda_*^k(\cC_*))$:
$$G^i(\HkG (\Lambda_*^k(\cC_*))  ) =\TiG{i} \circ \Lambda_{j_m,G}^{-1} ( \Lambda_{k,G}(H_{G_k}(\cC_*)) ) \subset \bTi{i} \circ  \bLambda^{k-j_m}(H_{G_k}(\cC_*)).$$

The set $\cE_1$ was chosen so that $ \bLambda(\Delta)  \subset \cE_1$, where $\Delta$ is as in Theorem $2$, and $\bLambda(\cE_1) \subset  \cE_1$. Therefore, 
$$ G^i(\HkG (\Lambda_*^k(\cC_*))  )  \subset   \bTi{i} (\cE_1).
$$

Now, to demonstrate the invariance of the set $\cE$, we verify that 


$$\bLambda(\cE_1) \Subset  \cE_1, \quad \bLambda(\cE_3 ) \Subset  \cE_1, \quad \bG( \cE_3) \Subset  \cE_4$$
(see \cite{GP} for programs used in this verification).

These inclusions imply (see Figure $\ref{containDynamics}$) that for any sequence $\{j_n\}_{q-1}^l$, $0 \le j_{q-1} < j_q < \ldots < j_l \le k-1$, the set $ \bTis{i}{q}{l} (\cE_1)$ is compactly contained in $\cE$. The set $\cE$ is depicted in Figure $\ref{BoundedFig}$.

\begin{figure}[t]
\psfrag{one}{$\cE_1$}
\psfrag{two}{$\cE_2$}
\psfrag{three}{$\cE_3$}
\psfrag{four}{$\cE_4$}
\psfrag{L}{$\Lambda$}
\psfrag{G}{$G$}
\begin{center}
\includegraphics[width=0.3 \textwidth]{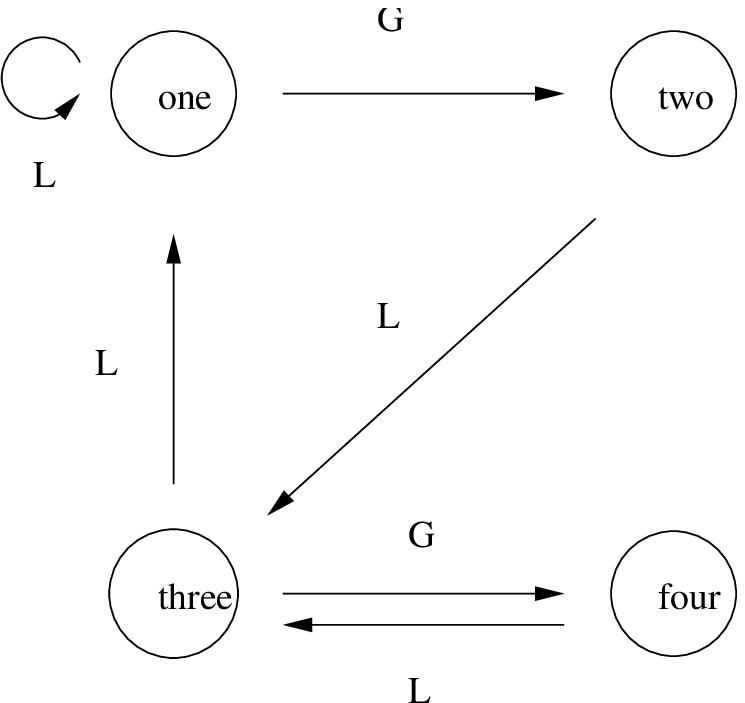} 
\caption{ Invariance of the set $\cE$ under the action of $\bT_n$.}\label{containDynamics}
\end{center}
 \end{figure}

\end{proof}

The next technical property of the set $\cE$ will be required in the proof of weak rigidity in Section $\ref{Weak_Rigidity}$.  This ``separation'' property has been verified directly on the computer.

\begin{lemma}\label{separation}
The projections of sets  $\cE_2$ and $\cE_4$ on the horizontal axis are separated from $0$ and from rescalings of themselves:  ${\rm dist}(\cP_x (\cE_2 \cup  \cE_4),0)$  is strictly positive, and 
$$|\lambda_-| \,  {\rm sup}_{p \in \cE_2 \cup \cE_4}|\cP_x p|<  {\rm inf}_{p \in \cE_2 \cup \cE_4}|\cP_x p|.$$ 
\end{lemma}

\begin{remark}\label{norms}
We have computed (see \cite{GP}) the upper and lower bounds on the norms of $\|D \bT_1 v\|$ and $\|D (\bG \circ \bLambda )\|$ to be as follows:
\begin{eqnarray}
\hspace{-2.0cm} &&\| D \bT_1 \|_{\cE_1} < 0.764 \equiv A_1, \quad \phantom{aaaaaaa} \| D \bT_1 \|_{\cE_3} < 0.344 \equiv A_3, \\
&&\| D \left(\bG \circ \bLambda \right) \|_{\cE} < 0.585 \equiv a, \quad  \inf_{v \in \fR^, \|v\|=1} \| D \bT_1 v \|_{\cE_1 \cup \cE_3} > 0.034\equiv b,
\end{eqnarray}
We will also denote 
$$A=\max\{A_1,A_3\}.$$

\end{remark}

\medskip

The next Lemma, albeit straightforward, will be important in our proofs of convergence of sets $\cC_G^k$ and existence of a bi-Lipschitz conjugacy between the limit sets.

\begin{lemma}\label{HkGLemma}
There exist $\varrho>0$ and a function $C(\varrho)$ with the property 
$$\lim_{\varrho \rightarrow 0} C(\varrho)=0,$$
such that for any $F \in \bW(\varrho)$

\begin{eqnarray}
 \arrowvert \cP_x \left( \HkG - \Lambda_{k,G} \circ \Lambda^{-k}_* \right)\arrowvert_{\Lambda^k_*(\cC_*)} &\le& C(\varrho) \, |\lambda_{k,G}| \, \nu^k, \\
 \arrowvert \cP_u \left( \HkG - \Lambda_{k,G} \circ \Lambda^{-k}_* \right)\arrowvert_{\Lambda^k_*(\cC_*)} &\le& C(\varrho) \, \mu_{k,G} \, \nu^k.
\end{eqnarray}

Furthermore, if  $F \in \bW_\omega(\varrho)$ then
\begin{eqnarray}
\label{omega_conv1} \arrowvert \cP_x \left( \HkG - \Lambda_{k,G} \circ \Lambda^{-k}_* \right)\arrowvert_{\Lambda^k_*(\cC_*)} &\le& C(\varrho) \, |\lambda_{k,G}| \, \omega^k, \\
\label{omega_conv2} \arrowvert \cP_u \left( \HkG - \Lambda_{k,G} \circ \Lambda^{-k}_* \right)\arrowvert_{\Lambda^k_*(\cC_*)} &\le& C(\varrho) \, \mu_{k,G} \, \omega^k.
\end{eqnarray}
\end{lemma}

\begin{proof}
By the strong structural stability property of the hyperbolic sets (see e.g. Theorem 18.1.3 and 18.2.1 in \cite{KH}), $\Arrowvert G_k - G_* \Arrowvert_{\cD_3}  \converge{{k \rightarrow \infty}}   0$ implies 
\begin{equation}\label{stab_converge}
\arrowvert H_{G_k} - Id \arrowvert_{\cC_*}  \converge{{k \rightarrow \infty}}   0,
\end{equation}
and, in fact,  if $\varrho$ is sufficiently small then there exists a constant $C'$ such that for all $F \in \bW(\varrho)$
$$\arrowvert H_{G_k} - Id \arrowvert_{\cC_*}  < C' \, \Arrowvert G_k - G_* \Arrowvert_{\cD_3}.$$
Now, for all $F \in \bW(\varrho)$
$$\Arrowvert G_k - G_* \Arrowvert_{\cD_3}  \le C''(\varrho) \,  \nu^k,$$
where the ``constant'' $C''(\varrho)$ decreases to zero with the size of the local manifold $\bW(\varrho)$, therefore,
\begin{equation}\label{unif_converge_nu}
\arrowvert H_{G_k} - Id \arrowvert_{\cC_*}  \le C(\varrho) \,  \nu^k
\end{equation}
for some function $C(\varrho)$ with the property $\lim_{\varrho \rightarrow 0} C(\varrho)=0$. In a similar way, if $F \in \bW_\omega(\varrho)$, then  
\begin{equation}\label{unif_converge_omega}
\arrowvert H_{G_k} - Id \arrowvert_{\cC_*}  \le C(\varrho)  \, \omega^k.
\end{equation}

Finally,
\begin{eqnarray}\label{HkGconv}
\nonumber \arrowvert \cP_x ( H_{G_k} - Id) \arrowvert_{\cC_*} &=&\arrowvert \cP_x (\Lambda_{k,G}^{-1} \circ \HkG \circ \Lambda_*^k - \Lambda_*^{-k} \circ \Lambda_*^k)  \arrowvert_{\cC_*}\\
&=&|\lambda_{k,G}|^{-1}  \arrowvert \cP_x ( \HkG  - \Lambda_{k,G} \circ \Lambda_*^{-k} )\arrowvert_{\Lambda_*^k(\cC_*)},
\end{eqnarray}
and similarly for $\cP_u ( H_{G_k} - Id)$. The claim follows.
\end{proof}

In  several following propositions and theorems we will have to use a number of ``constants'' $c_i(\varrho)$  all of which have the property 
$$\lim_{\varrho \rightarrow 0} c_i(\varrho)=0.$$

\begin{prop}\label{convergence}
 There exists $\varrho>0$ such that for all $F \in \bW(\varrho)$ the sets $\cV^k_G$ and $\cC^k_G$ converge in the Hausdorff metric, specifically:
$$d_H(\cV_G^k,\cV_G^{k+1}) \le {\rm const} \, \theta^k, \quad d_H(\cC_G^k,\cC_G^{k+1}) \le {\rm const} \, \theta^k,$$  
where $\theta=0.436$.

Furthermore, for any fixed $i$, there is $K>0$, such that for all $k \ge K$ 
\begin{equation}\label{convergence1}
d_H(\cV_G^{k,i},G^i((0,0))) \le {\rm const} \, \theta^k,
\end{equation}
the limit set is closed, and satisfies

$$\cC_G^{\infty} = \overline{\bigcup_{i \in \fZ} G^{i}((0,0))}.$$
\end{prop}

\begin{proof}
 Clearly, 
$$d_H(\cV_G^{k,0},\cV_G^{k+1,0})=d_H(\Lambda_{k,G}(\cU_{G_k}^k),\Lambda_{k+1,G}(\cU_{G_{k+1}}^{k+1})) \le {\rm const}\, |\lambda_-|^k.$$

Let the binary expansion of $i<2^k$ be as in $(\ref{binaryi})$. Recall, that according to Lemma $\ref{4-set}$

$$G^i(\Lambda_{k,G}(\cU_{G_k}^k)) \subset  \bTi{i} \circ \bLambda^{k-j_m} (\cU_{G_k}^k),$$ 
$$G^i(\Lambda_{k+1,G}(\cU_{G_{k+1}}^{k+1})) \subset   \bTi{i} \circ \bLambda^{k+1-j_m} (\cU_{G_{k+1}}^{k+1}),$$
where $\bTi{i}$ is as in $(\ref{Ttransform5})$.

Let $s_k \equiv \Lambda_{k,G}(s)  \in \cV_G^{k,0}$ and $p_{k+1} \equiv \Lambda_{k+1,G}(p) \in \cV_G^{k+1,0}$ be any two points in the corresponding sets.  Since $ j_m \le k-1$, both $\bLambda^{k-j_m} (s) \subset \bLambda^{k-j_m} (\cU_{G_k}^{k})$ and $\bLambda^{k+1-j_m} (p) \subset \bLambda^{k+1-j_m} (\cU_{G_{k+1}}^{k+1})$ are  contained in $\cE_1$. According to Lemma $\ref{4-set}$ the sequences
$$\bTis{i}{m-l+1}{m} \circ \bLambda^{k-j_m} (s) \quad {\rm and} \quad \bTis{i}{m-l+1}{m} \circ \bLambda^{k+1-j_m} (p), \quad 1 \le l \le m,$$
land in $\cE_1$ if $j_{m-l+1}-j_{m-l} >1$, and in  $\cE_3$ if $j_{m-l+1}-j_{m-l} = 1$.  Suppose, out of $m$ differences $j_{n}-j_{n-1}$, $n=1,\ldots,m$, $q$ are larger than $1$ and $m-q$ equal to $1$.  Then,
\begin{eqnarray}\label{diam_bound}
\nonumber  |G^i(s_k)-G^i (p_{k+1})| &\le& \left|  \bTi{i} \circ \bLambda^{k-j_m} (s) -\bTi{i} \circ  \bLambda^{k+1-j_m} (p) \right| \\
  &\le & |\lambda_-|^{k-m}  \|D \bT_1 \|_{\cE_1}^{q} \|D \bT_1 \|_{\cE_3}^{m-q}  |s-\bLambda(p)|,
\end{eqnarray}
where
\begin{eqnarray}
\nonumber \arrowvert  \bTi{i} \circ \bLambda^{k-j_m} (s) &-&\bTi{i} \circ  \bLambda^{k+1-j_m} (p) \arrowvert \equiv \\
\nonumber  &\equiv& \sup_{G \in \bG, \Lambda \in \bLambda} \left|  \TiG{i} \circ \Lambda^{k-j_m}  (s) -  \TiG{i} \circ \Lambda^{k+1-j_m}  (p) \right|.
\end{eqnarray}.

The more often $\|D \bG \|$in  $\|D \bT_1\|$ has to be evaluated on $\cE_1$, that is, the more often the bound $A_1$ (see Remark $\ref{norms}$) appears in the product in $(\ref{diam_bound})$, the worse the resultant  bound. Recall that $m\le k$ and $j_m \le {k-1}$. Therefore, if $m\le \left[{k \over 2} \right]$, then all differences $j_{n}-j_{n-1}$ may be larger than 1 ($q=m$), and 
$$ |G^i(s_k)-G^i (p_{k+1})| < |\lambda_-|^{k-m} A_1^{m} |s-\bLambda(p)| <{\rm const} \, |\lambda_-|^{\left[{k \over 2}\right]} A_1^{\left[{k \over 2}\right]}.$$

If  $m > \left[{k \over 2} \right]$ then there are at most $q=k-m$ differences $j_{n}-j_{n-1}$ that are larger than $1$:
\begin{eqnarray}
\nonumber |G^i(s_k)-G^i (p_{k+1})| &<& |\lambda_-|^{k-m}A_1^{k-m} A_3^{m-(k-m)}  |s-\bLambda(p)| \\
\nonumber & = & \left[{A_3^2 \over A_1 |\lambda_-|} \right]^m \,  \left[{|\lambda_-| A_1 \over A_3} \right]^k |s-\bLambda(p)|,
\end{eqnarray}
and since $A_3^2 / |\lambda_-| A_1 <1$ we get in this case
$$ |G^i(s_k)-G^i (p_{k+1})| < {\rm const} \, \left[{A_3^2 \over A_1 |\lambda_-|} \right]^{\left[k \over 2\right]} \left[{|\lambda_-| A_1 \over A_3} \right]^k= {\rm const} \, \left[|\lambda_-|^{1/2} A_1^{1/2 } \right]^k.$$

Since  $|\lambda_-|^{1/2} A_1^{1/2 } < 0.436 < 1$, we get
$$  |G^i(s_k)-G^i(p_{k+1})| < {\rm const } \,  0.436^k, \quad 1 \le i \le 2^k-1.$$

Any point in $\cV_G^{k,i}$ can be represented as  $G^i(s_k)$ for some  $s_k \in \cV_G^{k,0}$, and any point in $\cV_G^{k+1,i}$ can be represented as  $G^i(p_{k+1})$ for some  $p_{k+1} \in \cV_G^{k+1,0}$, therefore
$$ d_H( \cV_G^{k,i} , \cV_G^{k+1,i} ) < {\rm const } \, 0.436^k, \quad i \le 2^k-1.$$

A similar computation holds for inverse iterates $G^{-1}=T \circ G \circ T$:
$$ d_H( \cV_G^{k,2^k-n} , \cV_G^{k+1,2^{k+1}-n} ) < {\rm const } \, 0.436^k, \quad 1 \le n \le 2^k-1.$$

This demonstrates that the Hausdorff distance between components  $\cV_G^{k,i}$ and components $\cV_G^{k+1,i}$, $\cV_G^{k+1,2^k+i}$ decreases with $k$ at a  geometric rate.

An identical argument for sets $\cC_G^k$ (rather than $\cV_G^k$) shows that these sets converge in the Hausdorff metric at the same rate $\theta$. We define the set $\cC^\infty_G$ as the set of all limit points of sequences $\{p_k\}$,  $p_k \in \cC^k_G$. Such set is clearly closed. 

Finally, to show $(\ref{convergence1})$,  we again notice that if $s_{k,i} \in \cV^{k,i}$, then there exists a point $s_k \equiv \Lambda_{k,G}(s) \in \cV_G^{k,0}$ such that $s_{k,i}=G^i(s_k)$. Therefore, if $K \in \fZ$ is such that $2^K>i$ then for any $k > K$
$$|s_{k,i}-G^i ((0,0))|=|G^i(s_k)-G^i ((0,0))| \le  |\lambda_-|^{k-m}  \|D \bT_1 \|_{\cE_1}^{q} \|D \bT_1 \|_{\cE_3}^{m-q}  |s-(0,0)| < {\rm const} \, \theta^k. $$
\end{proof}

\medskip

We will now show that the set $\cC_G^\infty$ is invariant 
for $G$.

\begin{lemma}
For any $F \in \bW(\varrho)$ the sets $\cC^k_G$ are invariant under $G$. 
The same is true about the set $\cC_G^\infty$.
\end{lemma}
\begin{proof}
This follows from a simple computation:
\begin{eqnarray}
\nonumber G^{2^k}( \HkG(\Lambda_*^k(\cC_*)) )&=&G^{2^k}( 
\Lambda_{k,G}(\Lambda_{k,G}^{-1}(\HkG(\Lambda_*^k(\cC_*)))) )\\
\nonumber &=&G^{2^k}( \Lambda_{k,G}(H_{G_k}(\cC_*))) \\
\nonumber &=&\Lambda_{k,G}(G_k(H_{G_k}(\cC_*)))\\
\nonumber &\subset& \Lambda_{k,G}(H_{G_k}(\cC_*))\\
\nonumber &=& \HkG(\Lambda_*^k(\cC_*).
\end{eqnarray}

By Proposition $\ref{convergence}$, a point $p_\infty \in 
\cC_G^\infty$ is a limit point of some sequence $\{p_k\}$, 
$p_k \in \cC_G^k$. Because of the invariance of $\cC_G^k$ we have that 
$G(p_k) \in \cC_G^k$ for all $k\in \fN$. Analyticity of the map $G$ 
implies that $\{G(p_k)\}$ converges in $\cC_G^\infty$:
$$G(p_\infty)=G(\lim_{k \rightarrow \infty} (p_k))=\lim_{k \rightarrow 
\infty} G(p_k) \in \cC_G^\infty.$$
\end{proof}

\medskip

\begin{figure}[t]
\begin{center}
\resizebox{90mm}{!}{\includegraphics[angle=-90]{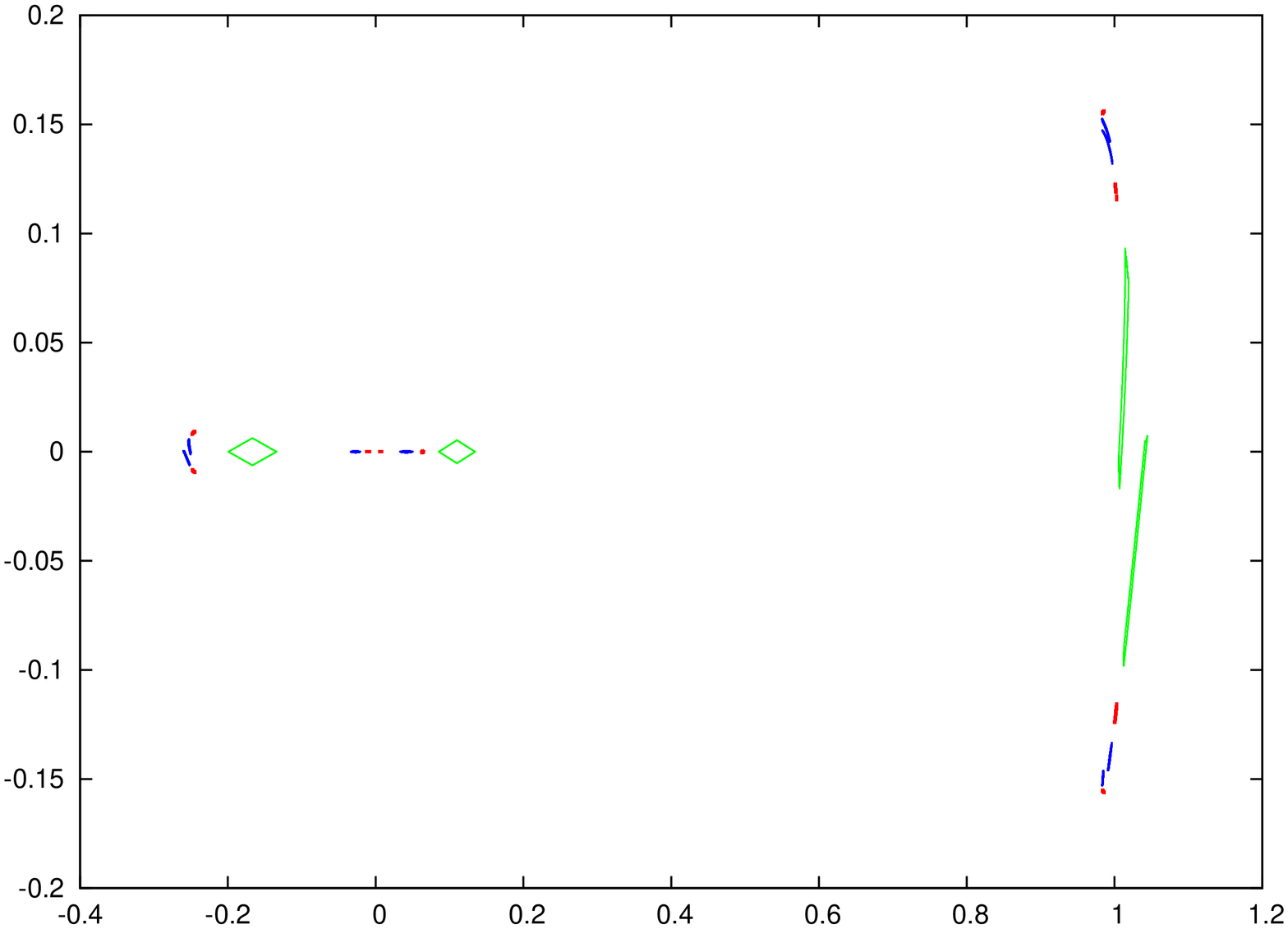}}
\caption{Approximations of sets $\cC_G^1$ (green), $\cC_G^2$ (blue) and $\cC^3_G$ (red)}\label{CkFig}
\end{center}
\end{figure}

We will now address the convergence properties of transformations $\iHkG$.

\begin{prop}\label{Hconvergence}
 There exists $\varrho>0$ such that for all $F \in \bW(\varrho)$ the following holds.

\noindent 1)  The transformations $G^i \circ \Lambda_{k,G} \circ \Lambda^{-k}_* \circ G_*^{-i}$  are defined and analytic on $\cV_*^{k,i}$ for all $k \in \fN$ and  $0 \le i < 2^k$, and satisfy
\begin{eqnarray}
\label{iHkGdiff} \arrowvert \iHkG - G^i \circ \Lambda_{k,G} \circ \Lambda^{-k}_* \circ G_*^{-i}\arrowvert_{\cC_*^{k,i}} &\le& C(\varrho)\, (\nu \, \theta)^k, \\
\label{ApprConjDiff}\arrowvert \, G^i \circ \Lambda_{k,G} \circ \Lambda^{-k}_* \circ G_*^{-i} - Id \, \arrowvert_{\cV_*^{k,i}}  &\le& c_1(\varrho),
\end{eqnarray}
where  $C(\varrho)$ is as in Lemma $\ref{HkGLemma}$, and   $c_1(\varrho)$ is some function of $\varrho$ independent of $k$, $i$ and $G$, and satisfying
$$\lim_{\varrho \rightarrow 0} c_1(\varrho)=0.$$

\medskip

\noindent 2) For any $p \in \cV^{k,i}_*$ and $s \in \cV^{{k+1},i}_*$ 
\begin{equation}
\label{iHkGconv} \arrowvert G^i \circ \Lambda_{k,G} \circ \Lambda^{-k}_* \circ G_*^{-i}(p) - G^i \circ \Lambda_{k+1,G} \circ \Lambda^{-k-1}_* \circ G_*^{-i}(s)\arrowvert < {\rm const} \, \theta^k.
\end{equation}
\end{prop}

\begin{proof}$$\phantom{a}$$
\noindent 1) Notice that 
$$G^i \circ \Lambda_{k,G} \circ \Lambda^{-k}_* \circ G_*^{-i}({G^i_*(\cV_*^{k,0})})= G^i  ( {\Lambda_{k,G}(\cU_*^k)} ) \subset  G^i ( {\bLambda^k(\cU_*^k)} ).$$ 
By Lemma $\ref{4-set}$ the iterate $G^i$, $1 \le i < 2^k$, is well-defined and analytic on $\bLambda^k(\cU_*^k)$ for all $G \in \bG$. 

 Proving  $(\ref{iHkGdiff})$ is similar to $(\ref{diam_bound})$ and arguments that follow it:
\begin{eqnarray}
\nonumber \arrowvert \iHkG \!\!\!\!&-&\!\!\!\! G^i \circ \Lambda_{k,G} \circ \Lambda^{-k}_* \circ G_*^{-i}\arrowvert_{G^i_*(\Lambda_*^k(\cC_*))} \\
\nonumber &=& \arrowvert G^i \circ H_{k,G} - G^i \circ \Lambda_{k,G} \circ \Lambda^{-k}_*\arrowvert_{\Lambda_*^k(\cC_*)} \\
\nonumber & \le &  \left|\TiG{i} \circ \Lambda_{j_m,G}^{-1} \circ \Lambda_{k,G}  \circ \Lambda_{k,G}^{-1} \circ  H_{k,G} -\TiG{i} \circ \Lambda_{j_m,G}^{-1} \circ \Lambda_{k,G} \circ \Lambda^{-k}_*\right|_{\Lambda_*^k(\cC_*)} \\
\nonumber & \le & \theta^k  \left| \Lambda_{k,G}^{-1} \circ  H_{k,G} \circ \Lambda_*^k - Id \right| _{\cC_*} \le C(\varrho) \, (\nu \, \theta)^k,
\end{eqnarray}
where the function $C(\varrho)$ is as in Lemma $\ref{HkGLemma}$.

We will now demonstrate $(\ref{ApprConjDiff})$ in two steps.

\noindent {\it \underline{Step (1)}}. Write
\begin{equation}\label{bracket}
G^i \circ \Lambda_{k,G} \circ \Lambda^{-k}_* \circ G_*^{-i}\!=\cJ_{G,m,i} \circ \!\left\{ G_{j_m}\!\circ \Lambda_{j_m,G}^{-1} \circ \Lambda_{k,G} \circ \Lambda_*^{j_m\!-k} \circ G_*^{-1}\!\right\} \!\circ \cJ_{G_*,m,i}^{-1},
\end{equation}
where we have denoted for all $1 \le q \le m$:
\begin{eqnarray}
\nonumber \cJ_{G,q,i} &\equiv& \TisG{i}{1}{q-1} \circ \Lambda_{j_{q-1},G}^{-1} \circ \Lambda_{j_q,G}=G^{i} \circ  \Lambda_{j_m,G} \circ  \TisG{i}{q+1}{m}^{-1}  \circ G_{j_q}^{-1}.
\end{eqnarray}

The image of $\cV_*^{k,i}$ under the inverse of this map is contained in $\cE_2 \cup \cE_4$ for all $1 \le q \le m$:
\begin{equation}\label{cJimage}
\cJ_{G_*,q,i}^{-1}(\cV_*^{k,i})= G_* \circ \Tis{i}{q+1}{m} \circ \Lambda^{k-j_m}_*( \cU_*^{k}) \subset \bG \circ \bTis{i}{q+1}{m} \circ \bLambda^{k-j_m}( \cU_*^{k}) \Subset \cE_2 \cup \cE_4,
\end{equation}
while
\begin{equation}\label{cJimage2}
G_{*}^{-1}(\cJ_{G,q,i}^{-1}(\cV_*^{k,i})) \Subset \cE_1 \cup \cE_3.
\end{equation}

Since $|\lambda_{G_n} - \lambda_*| \le  c_2(\varrho) \, \nu^n$ and  $|\mu_{G_n} - \mu_*| \le c_2(\varrho) \, \nu^n$,  we get
\begin{equation}\label{LinId}
\arrowvert \Lambda_{j_m,G}^{-1} \circ \Lambda_{k,G} \circ \Lambda_*^{j_m-k}-Id \arrowvert_{G^{-1}_*(\cJ_{G_*,m,i}^{-1}(\cV_*^{k,i}))} \le c_3(\varrho) \, \nu^{j_m}.
\end{equation}
Since containment of  $G_{j_q}^{-1}(\cJ_{G,q,i}^{-1}(\cV_*^{k,i}))$ in $\cE_1 \cup \cE_3$ is compact, it is possible to  choose $\varrho$ so that for all $1 \le q \le m$
$$\Lambda_{j_q,G}^{-1} \circ \Lambda_{k,G} \circ \Lambda_*^{j_q-k} (G^{-1}_*(\cJ_{G_*,m,i}^{-1}(\cV_*^{k,i})))  \subset \cE_1 \cup \cE_3.$$ 
The map $G_{j_m}$ is defined and analytic on $\cE_1 \cup \cE_3$ and maps it into $\cE_2 \cup \cE_4$, and therefore $ G_{j_m} \circ \Lambda_{j_m,G}^{-1} \circ \Lambda_{k,G} \circ \Lambda_*^{j_m-k} \circ G_*^{-1}$ is analytic on $\cJ_{G_*,m,i}^{-1}(\cV_*^{k,i})$ and maps it into $\cE_2 \cup \cE_4$. Because of $(\ref{LinId})$ we also have for any $n>j_q$:
\begin{eqnarray}\label{IdDiffTr}
\nonumber \arrowvert Id \!\!\!\!&-&\!\!\!\!  G_{j_q} \circ \Lambda_{j_q,G}^{-1} \circ \Lambda_{n,G} \circ \Lambda_*^{j_q-n} \circ G^{-1}_* \arrowvert_{\cJ_{G_*,q,i}^{-1}(\cV_*^{k,i})} \\
\nonumber &\le& \arrowvert  G_{j_q} \circ \Lambda_{j_q,G}^{-1} \circ \Lambda_{n,G} \circ \Lambda_*^{j_q-n} \circ G^{-1}_*-G_{j_q} \circ G_*^{-1} \arrowvert_{\cJ_{G_*,q,i}^{-1}(\cV_*^{k,i})} \\
\nonumber &+& \arrowvert G_{j_q} \circ G_*^{-1} -Id \arrowvert_{\cJ_{G_*,q,i}^{-1}(\cV_*^{k,i})}\\
\hspace{-1.5cm} &\le& {\rm const} \, c_3(\varrho) \, \nu^{j_q} + c_4(\varrho) \, \nu^{j_q} \le c_5(\varrho) (1+{\rm const})\nu^{j_q}=c_6(\varrho) \nu^{j_q},
\end{eqnarray}
where $c_4(\varrho)$ is another constant decreasing to zero together with $\varrho$, and $c_5(\varrho)$ is the maximum of $c_3(\varrho)$ and $c_4(\varrho)$.

As the result of the above discussion, we have
\begin{equation}\label{Gjm}
G^i \circ \Lambda_{k,G} \circ \Lambda^{-k}_* \circ G_*^{-i} \arrowvert_{\cV^{k,i}_* }= \cJ_{G,m,i} \circ  \left\{ Id+h_{j_m} \right\} \circ \cJ_{G_*,m,i}^{-1} \arrowvert_{\cV^{k,i}_* },
\end{equation}
where $h_{j_m}$ is some function analytic on $\cJ_{G_*,m,i}^{-1}({\cV^{k,i}_* })$ and satisfying 
\begin{equation}\label{hjm}
\arrowvert h_{j_m}  \arrowvert_{\cJ_{G_*,m,i}^{-1}(\cV^{k,i}_*) } \le c_6(\varrho) \nu^{j_m}.
\end{equation}

\noindent {\it \underline{Step (2)}}. At the next step, to obtain the bound $(\ref{ApprConjDiff})$ we will use an inductive argument.

Suppose that for $q \le m$
\begin{equation}\label{Gjq}
 G^i \circ \Lambda_{k,G} \circ \Lambda^{-k}_* \circ G_*^{-i} \arrowvert_{\cV^{k,i}_* }= \cJ_{G,q,i} \circ  \left\{ Id+h_{j_q} \right\} \circ \cJ_{G_*,q,i}^{-1} \arrowvert_{\cV^{k,i}_* },
\end{equation}
where $h_{j_q}$ is some function analytic on $\cJ_{G_*,q,i}^{-1}({\cV^{k,i}_* })$ and satisfying 
$$\arrowvert h_{j_q}  \arrowvert_{\cJ_{G_*,q,i}^{-1}(\cV^{k,i}_*) } \le c_6(\varrho) \left[\sum_{i=0}^{m-q} a^i \nu^{j_{q+i}-j_q}   \right] \nu^{j_q}.$$
This is certainly satisfied for $q=m$ (see $(\ref{Gjm})$ and $(\ref{hjm})$).

We prove that a representation similar to $(\ref{Gjq})$ holds for $q-1$ with a similar bound on $h_{j_{q-1}}$. First,
\begin{eqnarray}
\nonumber   G^i \circ \Lambda_{k,G} \circ \Lambda^{-k}_* \circ G_*^{-i} \arrowvert_{\cV^{k,i}_* } = \cJ_{G,q-1,i} & \circ &  \left\{ G_{j_{q-1}} \circ \Lambda_{j_{q-1},G}^{-1} \circ \Lambda_{j_q,G} \circ (Id+h_{j_q})  \right.\\
&\circ& \left. \Lambda_*^{j_{q-1}-j_{q}} \circ G^{-1}_* \right\} \circ \cJ_{G_*,q-1,i}^{-1} \arrowvert_{\cV^{k,i}_* }.
\end{eqnarray}

Again, consider the map in the brackets:
\begin{eqnarray}\label{IdDiffTr2}
\nonumber \arrowvert Id \!\!\!\!&-&\!\!\!\! G_{j_{q-1}} \circ \Lambda_{j_{q-1},G}^{-1} \circ \Lambda_{j_q,G} \circ (Id+h_{j_q})  \circ  \Lambda_*^{j_{q-1}-j_q} \circ G^{-1}_* \arrowvert_{ \cJ_{G_*,m-1,i}^{-1}(\cV^{k,i}_*) } \\
\nonumber  &\le & \arrowvert Id -G_{j_{q-1}} \circ \Lambda_{j_{q-1},G}^{-1} \circ \Lambda_{j_q,G} \circ  \Lambda_*^{j_{q-1}-j_q} \circ G^{-1}_* \arrowvert_{ \cJ_{G_*,q-1,i}^{-1}(\cV^{k,i}_*)} \\
\nonumber& +&  \arrowvert G_{j_{q-1}} \circ \Lambda_{j_{q-1},G}^{-1} \circ \Lambda_{j_q,G} \circ (Id+h_{j_q})  \circ  \Lambda_*^{j_{q-1}-j_q} \circ G^{-1}_* \\
&-& G_{j_{q-1}} \circ \Lambda_{j_{q-1},G}^{-1} \circ \Lambda_{j_q,G} \circ  \Lambda_*^{j_{q-1}-j_q} \circ G^{-1}_* \arrowvert_{ \cJ_{G_*,q-1,i}^{-1}(\cV^{k,i}_*) }.
\end{eqnarray}

The first norm in $(\ref{IdDiffTr2})$ has been estimated in $(\ref{IdDiffTr})$. To provide a bound on the second norm we will use the fact that 
$$G_{j_{q-1}} \circ \Lambda_{j_{q-1},G}^{-1} \circ \Lambda_{j_q,G}= 
\left\{ G_{j_{q-1}} \circ \Lambda_{G_{j_q-1}} \right\} \circ \left[ \Lambda_{j_{q-1},G}^{-1} \circ \Lambda_{j_q-1,G}\right],$$
and that if $j_q-1-j_{q-1}=0$ then 
$$\Lambda_{j_{q-1},G}^{-1}( \Lambda_{j_q-1,G} (\cE_2 \cup \cE_4))=\cE_2 \cup \cE_4,$$ 
if  $j_q-1-j_{q-1}=1$ then  
$$\Lambda_{j_{q-1},G}^{-1}( \Lambda_{j_q-1,G} (\cE_2 \cup \cE_4)) \subset \cE_3,$$
if  $j_q-1-j_{q-1}\ge 2$ then  
$$\Lambda_{j_{q-1},G}^{-1}( \Lambda_{j_q-1,G} (\cE_2 \cup \cE_4)) \subset \cE_1$$
Therefore,
\begin{eqnarray}
\nonumber \arrowvert Id \!\!\!\!&-&\!\!\!\! G_{j_{q-1}} \circ \Lambda_{j_{q-1},G}^{-1} \circ \Lambda_{j_q,G} \circ (Id+h_{j_q})  \circ  \Lambda_*^{j_{q-1}-j_m} \circ G^{-1}_* \arrowvert_{ \cJ_{G_*,m-1,i}^{-1}(\cV^{k,i}_*) } \\
\nonumber&\le&  c_6(\varrho)\nu^{j_{q-1}}+ |\lambda_-|^{j_q-1-j_{q-1}} \left\| D \left\{G_{j_{q-1}}\! \circ \Lambda_{G_{j_q-1}}\right\} \right\|_{\cE} \arrowvert h_{j_q} \arrowvert_{ \cJ_{G_*,q,i}^{-1}(\cV^{k,i}_*) } \\
\nonumber&\le&  c_6(\varrho)\nu^{j_{q-1}}+ a  c_6(\varrho) \left[\sum_{i=0}^{m-q} a^i \nu^{j_{q+i}-j_q}   \right] \nu^{j_q} \\
\nonumber &=& c_6(\varrho) \left[\sum_{i=0}^{m-q+1} a^i \nu^{j_{q-1+i}-j_{q-1}}   \right] \nu^{j_{q-1}}.
\end{eqnarray}
Therefore,
\begin{eqnarray}\label{Gjq2}
\nonumber   \arrowvert Id  \!\!\!\!&-&\!\!\!\! G^i \circ \Lambda_{k,G} \circ \Lambda^{-k}_* \circ G_*^{-i} \arrowvert_{\cV^{k,i}_* } \\
\nonumber &=& \arrowvert Id - \cJ_{G,1,i} \circ  \left\{ Id+h_{j_1} \right\} \circ \cJ_{G_*,1,i}^{-1} \arrowvert_{\cV^{k,i}_* } \\
\nonumber &=&\arrowvert Id \!-\! \Lambda_{j_1,G} \circ  \! \left\{ Id\!+\!h_{j_1} \right\} \! \circ \Lambda_*^{-j_1} \arrowvert_{\cV^{k,i}_*} \\ \nonumber &=& \arrowvert Id \!-\! \Lambda_{j_1,G} \circ  \Lambda_*^{-j_1} \arrowvert_{\cV^{k,i}_*} \!+\! \arrowvert \Lambda_{j_1,G} \circ h_{j_1}  \arrowvert_{ \cJ_{G_*,1,i}^{-1}(\cV^{k,i}_*\!) } \\
\nonumber & \le &c_3(\varrho)  + |\lambda^{j_1}_-|  c_6(\varrho) \left[\sum_{i=0}^{m-1} a^i \nu^{j_{1+i}-j_1}   \right] \nu^{j_1} \\ 
\nonumber &\le& c_3(\varrho) + c_6(\varrho) \, |\lambda_- \nu|^{j_1}  \left[\sum_{i=0}^{k-1} (a \nu)^i   \right] \\
\nonumber & \le &c_3(\varrho) + c_6(\varrho) [ 1-a \nu  ]^{-1} \equiv c_1(\varrho),
\end{eqnarray}
the last equality being the definition of $c_1(\varrho)$.

\medskip

\noindent 2) 
To demonstrate $(\ref{iHkGconv})$ we notice that $G^{-i}_*(p)$ is in  $\cV^{k,0}_*$ and $G^{-i}_*(s)$ is in $\cV^{k+1,0}_*$, while  $\hat{p}=\Lambda^{-k}_*(G^{-i}_*(p))$ and $\hat{s}=\Lambda^{-k-1}_*(G^{-i}_*(s))$ are in $\cU_*^k$. It follows from a computation similar to $(32)$ that 
\begin{eqnarray}
\nonumber &&\arrowvert G^i \circ \Lambda_{k,G} \circ \Lambda^{-k}_* \circ G_*^{-i}(p) -G^i \circ \Lambda_{k+1,G} \circ \Lambda^{-k-1}_* \circ G_*^{-i}(s)\arrowvert =\\
\nonumber && \arrowvert G^i \circ \Lambda_{k,G}(\hat{p}) -G^i \circ \Lambda_{k+1,G}(\hat{s})\arrowvert \\
\nonumber  && \le  \left|  \bTi{i} \circ \bLambda^{k-j_m} (\hat{p}) -\bTi{i} \circ  \bLambda^{k+1-j_m} (\hat{s}) \right|\\
\nonumber && \le  \ {\rm const} \, \theta^k. 
\end{eqnarray}
\end{proof}

\medskip

The above proposition implies, that if $p$ is in the limit set $\cC_*^{\infty}$, then there exist integers $i$ and $K$, dependent on $p$, and a  sequence of points $p_{k,i} \in \cC_*^{k,i}$, $k \ge K$, that converge to $p$: $\lim_{k \rightarrow \infty} p_{k,i}=p$.  We have from $(\ref{Hconj})$
$$G \circ \cH_{k,G}(p_{k,i})=\cH_{k,G} \circ G_*(p_{k,i}).$$
Bounds $(\ref{iHkGdiff})$---$(\ref{iHkGconv})$ imply that the limit
\begin{equation}\label{Hconv}
\cH_G(p) \equiv \lim_{k \rightarrow \infty}\cH_{k,G}(p_{k,i})=\lim_{k \rightarrow \infty}  G^i \circ \Lambda_{k,G} \circ \Lambda^{-k}_* \circ G_*^{-i}(p_{k,i})
\end{equation}
exists.

We will finally demonstrate that the limit set $\cC_G^\infty$ is stable.

Recall, the definition of the upper Lyapunov exponent of $(p,v) \in (\cD
\cap \fR^2) \times \fR^2$ with respect to $G$:
$$\chi(p,v;G) \equiv \overline{\lim}_{i \rightarrow \infty} {1 \over i}
\log\left[  \Arrowvert DG^i(p) v
\Arrowvert\right],$$
where $\Arrowvert \Arrowvert$ is some norm in $\fR^2$.
The maximal Lyapunov exponent of $p \in (\cD \cap \fR^2)$ with respect
to $G$ is defined as
$$\chi(p;G) \equiv \sup_{||v||=1} \chi(p,v;G).$$

\begin{lemma}
For any $F \in \bW(\varrho)$ and $p \in \cC^k_G$ the maximal Lyapunov exponent
$\chi(p;G)$ satisfies

$$ \chi(p;G) \le C {1\over 2^k},$$
where $C=C(G)$ is some constant independent of $k$.
\end{lemma}
\begin{proof}
Let $i=q 2^k +n$, $n=2^{j_1}+2^{j_2}+ \ldots + 2^{j_m} < 2^k$ and  $p
\in \HkG (\Lambda_*^k(\cC_*))$. Denote

\begin{equation}
\nonumber t \equiv  \Lambda^{-1}_{k,G}(p) \in \cC_{G_k}, \quad
 s \equiv \Lambda_{k,G}(G^q_k(t)) \in
\Lambda_{k,G}(\cC_{G_k}),
\end{equation}
then
\begin{eqnarray}
\nonumber 
DG^i(p)&=&DG^{n+q 2^k}(p)=DG^n(G^{q 2^k}(p)) \cdot DG^{q 2^k}(p)\\
\nonumber &=&DG^n( \Lambda_{k,G} \circ G^q_k \circ \Lambda^{-1}_{k,G}(p)  )\cdot 
\Lambda_{k,G} \cdot DG^q_k
(\Lambda^{-1}_{k,G}(p))  \cdot \Lambda^{-1}_{k,G}\\
\nonumber &=&DG^n( s ) \cdot \Lambda_{k,G} \cdot DG^q_k (t) \cdot 
\Lambda^{-1}_{k,G}\\
\nonumber 
&=& D \left( \TiG{n} \circ  \Lambda_{j_{m},G}^{-1} \right)(s)  \cdot
\Lambda_{k,G} \cdot  DG^q_k (t) \cdot \Lambda^{-1}_{k,G}
\end{eqnarray}
where  we have used the representation $(\ref{Gi})$. According to Lemma
$\ref{4-set}$
$$\TisG{n}{l+1}{m} \circ \Lambda_{j_m,G}^{-1} (s) \in \cE_1 \cup \cE_3.$$
Denote $B_k$ -  an upper bound  on the derivative norm
of $G_k$ on its invariant set  $H_{G_k}(\cC_*)$.
Then
\begin{equation}
\Arrowvert DG^i(p) \Arrowvert  \le \left(\frac{A}{|\lambda_-|}\right)^m 
\left( {
|\lambda_-| \over \mu_-  }\right)^{j_m} B^q_k \left( {
|\lambda_-| \over \mu_-  }\right)^k.
\end{equation}

Finally,
\begin{eqnarray}
\nonumber \chi(p;G) & = &
\overline{\lim}_{i \rightarrow\infty} {1 \over i} \log \left[  \Arrowvert DG^i(p)\Arrowvert \right]  \le  \lim_{i \rightarrow \infty} {1 \over i}
\log \left[  \left(\frac{A}{|\lambda_-|}\right)^m  B^q_k \left( {|\lambda_-| \over\mu_- }\right)^{j_m+k} \right] \\
\nonumber & \le & \lim_{i \rightarrow \infty} \left\{ {k\over i} \log\left[ A\left(\frac{|\lambda_-|}{\mu_-^2}\right) \right] + {q \over i} 
\log B_k \right\} \le {1 \over
2^k}   \log B_k.
\end{eqnarray}
\end{proof}

\medskip

Clearly, the above result implies stability of the limit set:

\begin{corollary}
For any $F \in \bW(\varrho)$ and $p \in \cC^\infty_G$ the maximal Lyapunov exponent
$ \chi(p;G)$ is equal to zero.
\end{corollary}

\medskip

\section{The stable set as a Cantor set}\label{Cantor_Set}

In this Section we will sketch the construction of the stable set using the method of presentation functions. The construction of this Section is almost identical to that of \cite{dCLM}, and we will therefore omit many  details. In fact, we will attempt to use the notation similar to that of  \cite{dCLM}.

Given $F \in \bW(\varrho)$ define the two {\it presentation functions}
$$\psi^F_0=\Lambda_F \quad {\rm and} \quad \psi^F_1=F \circ \psi^F_0,$$
and
\begin{eqnarray}
\Psi^{F,1}_0 &\equiv& \psi^F_0 \quad {\rm and} \quad \Psi^{F,1}_1 \equiv \psi^F_1,\\
\Psi^{F,n}_w&\equiv& \psi^F_{w_1} \circ \ldots \circ \psi^{R^n[F]}_{w_n}, \quad w=(w_1,\ldots,w_n) \in \{0,1\}^n.
\end{eqnarray}

\begin{lemma} For every  $F \in \bW(\varrho)$ there exists a simply connected closed set $B_F$ such that $B^1_0(F) \equiv \psi^F_0(B_F) \subset B_F$ and $B^1_1(F) \equiv \psi^F_1(B_F) \subset B_F$ are disjoint,  $F(B^1_1(F)) \cap B^1_0 \ne \emptyset$, and
\begin{equation}
\max\{ \|D \psi^F_0\|_{B_F}, \| D \psi^F_1 \|_{B_F}\} \le \vartheta, \quad \vartheta=0.272.
\end{equation}
 \end{lemma}

\begin{proof}
First, we verify the following on the computer:
$$
\psi^F_0(\tilde{B}) \subset \tilde{B}, \quad \psi^F_1(\tilde{B}) \subset \tilde{B} \quad {\rm and} \quad F( \psi^F_1(\tilde{B})) \cap \psi^F_0(\tilde{B}) \ne \emptyset
$$
for all $F \in \bW(\varrho)$, where
$$
\tilde{B}=\{(x,u) \in \fR^2: {(x-0.47)^2 \over 0.82^2}+{(u+0.04)^2 \over 0.301398806^2} \le 1\}.
$$ 
We also check that the sets $\psi^F_0(\tilde{B}) \subset \tilde{B}$ and $\psi^F_1(\tilde{B}) \subset \tilde{B}$ are disjoint.

Second, we verify that the boundary of the ellipse $\hat{B} \subset \tilde{B}$,
$$
\hat{B}=\{(x,u) \in \fR^2: {(x-0.47)^2 \over 0.53^2}+{u^2 \over 0.002370226^2} \le 1\},
$$ 
intersects each of $\psi^F_0(\tilde{B})$,  $\psi^F_1(\tilde{B})$ along a single arc. Therefore, the set 
$$B \equiv \psi^F_0(\tilde{B}) \cap \psi^F_1(\tilde{B}) \cap \hat{B}$$
(see Figure $\ref{cantor_s}$) is simply connected, and satisfies the claim. 
\end{proof}

\begin{figure}[t]\label{cantor_s}
\begin{center}
\resizebox{90mm}{!}{\includegraphics[angle=-90]{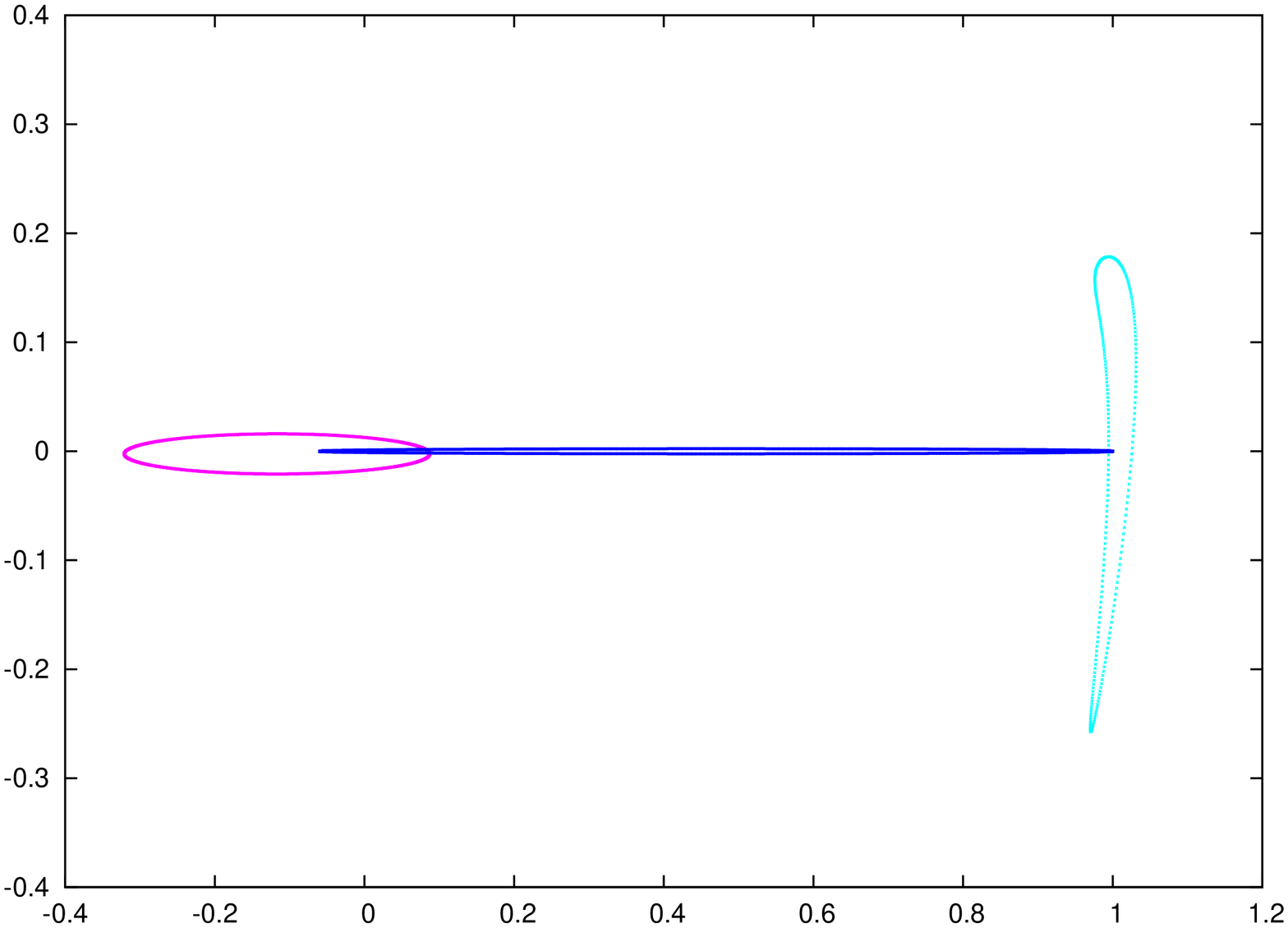}}
\caption{ Sets $\psi^F_0(\tilde{B})$ (magenta),  $\psi^F_1(\tilde{B})$ (cyan) and $\hat{B}$ (blue).}
\end{center}
\end{figure}

Set $B^1_0(F)=\psi^F_0(B_F)$,  $B^1_1(F)=\psi^F_1(B_F)$, and define ``pieces''
$$B^n_w(F)=\Psi^{F,n}_w(B_F), \quad w \in \{0,1\}^n.$$

One can view $\{0,1\}^n$ as an additive group of residues {\it mod} $2^n$ via an identification
$$w \rightarrow \sum_{k=0}^{n-1} w_{k+1} 2^k.$$
Let $p: \{0,1\}^n \rightarrow  \{0,1\}^n$, be the operation of adding $1$ in this group. The following Lemma has been proved in \cite{dCLM}, and it's proof holds in our case of area-preserving maps word by word:
\begin{lemma} \label{pieces}
$$\phantom{aaa}$$
\begin{itemize}
\item[1)] The above families of pieces are nested:
$$B^n_{wv} \subset B^{n-1}_w, \quad w \in \{0,1\}^{n-1}, \quad v \in  \{0,1\}.$$
\item[2)] The pieces $B^n_w, \quad w \in \{0,1\}^n$ are pairwise disjoint. 

\item[3)] $F$ permutes the pieces as follows: $F(B^n_w)=B^n_{p(w)}$ unless $p(w)=0^n$. If $p(w)=0^n$, then $F(B^n_w) \cap B^n_{0^n} \ne \emptyset$.

\item[4)] ${\rm diam}( B_w^n) \le {\rm const} \, \vartheta^n$.

\item[5)] $C_d^H(\cC^\infty_F) \le -{\log(2) / \log(\vartheta)} < 0.5324$, where
\begin{equation}\label{CS}
\cC^\infty_F \equiv \bigcap_{n=1}^\infty \bigcup _{w \in \{0,1\}^n} B_w^n.
\end{equation}
\end{itemize}
\end{lemma}

Since the set $\tilde{B}$ from Lemma 7.1 contains $(0,0)$, so does each piece $B^n_{0^n}$. It follows from part 3) of Lemma 7.2 that the set $\bigcup _{w \in \{0,1\}^n} B_w^n$ contains iterates $G^i((0,0))$ up to order $2^n$. Therefore, the Cantor set $\cC_F^\infty$ is the closure of the orbit of zero, and is equal to $\cC^\infty_G \cup F(\cC^\infty_G) \cup F(F(\cC^\infty_G))$.

Recall the definition $(\ref{dyadic_group})$  of the dyadic group.  Lemma $\ref{pieces}$ implies the following:
\begin{corollary}
The restriction $F \arrowvert_{\cC^\infty_F}$ is homeomorphic to $p: \{0,1\}^\infty \rightarrow  \{0,1\}^\infty$ via $h: \{0,1\}^\infty \rightarrow \cC^\infty_F$ defined as
$$h(w)=\bigcap_{n=1}^\infty B^n_{w_1 w_2 \ldots w_n}.$$
\end{corollary}

\medskip

\section{``Weak'' rigidity}\label{Weak_Rigidity}

In this Section we will demonstrate that the map $\cH_G$ is bi-Lipschitz for a subset of infinitely renormalizable maps.

\medskip

\begin{prop}\label{Lipschitz}
There exist $\varrho>0$ and $\omega$, 
$$\omega=\min \left\{  {\mu_* \over |\lambda_*|},  {b \over A} \right\},$$
such that for all $F \in \bW_\omega(\varrho)$ the transformation $\cH_G$ is bi-Lipschitz with a constant $\cL=\cL(\varrho)$, that satisfies $\cL(\varrho) \converge{{\varrho \rightarrow 0}} 1$.
\end{prop}
\begin{proof}
`Let $i=2^{j_1} + \ldots + 2^{j_m}$ and $\hat{i}=2^{\hat{j}_1}+\ldots +2^{\hat{j}_n}$, $i \ne \hat{i}$, be arbitrary but fixed.  Let $\{p_{k,i}^*\}_{k=\max(j_m,\hat{j}_n)}^\infty$ and  $\{s_{k,\hat{i}}^*\}_{k=\max(j_m,\hat{j}_n)}^\infty$ be any two sequences of points that satisfy:   $p_{k,i}^* \in \cC^{k,i}_*$ and  $s_{k,\hat{i}}^* \in \cC^{k,\hat{i}}_*$,  $p_{k,i}^* \ne s_{k,\hat{i}}^*$.

We would like to show, that there exist $\varrho>0$, $\omega<\nu$ and  $\cL=\cL(\varrho)$ , such that if $F \in \bW_\omega(\varrho)$ then the distances 
$$|\iHkG(p_{k,i}^*)-\!\!\phantom{a}_{\hat{i}}H_{k,G}(s_{k,\hat{i}}^*)| \quad {\rm and} \quad |p_{k,i}^*-s_{k,\hat{i}}^*|$$
are commensurate with a constant $\cL(\varrho)$, independent of $k$ and approaching $1$ as $\varrho \rightarrow 0$.

Commensurability, together with convergence property $(\ref{Hconv})$ implies that the limit $\cH_G$ is a bi-Lipschitz transformation.

Define the following points:

$$p_{k,i}=\iHkG(p_{k,i}^*),$$ 
$$s_{k,\hat{i}}=\hiHkG(s_{k,\hat{i}}^*),$$
$$p_{k}^*=G^{-i}_*(p_{k,i}^*) \equiv \Lambda_*^{k} (p^*),$$
$$s_{k}^*=G^{-\hat{i}}_*(s_{k,\hat{i}}^*)\equiv \Lambda_*^{k} (s^*),$$
$$p_{k,i}\equiv G^i(\Lambda_{k,G} (p))\equiv G^i(\Lambda_{k,G} (H_{G_k}(p^*))),$$ 
$$s_{k,\hat{i}}\equiv G^{\hat{i}} (\Lambda_{k,G} (s))\equiv  G^{\hat{i}}( \Lambda_{k,G} (H_{G_k}(s^*))),$$
%
%
where the the last four lines are understood as definitions of points $p$, $s \in \cC_{G_k}$ and $p^*$, $s^* \in \cC_*$.

For any $j<k$ and $F \in \bW_\omega(\varrho)$ there exists $c'_7(\varrho)$ such that $|\lambda_{G_{j}}| \le |\lambda_*|+c'_7(\varrho) \omega^j$, therefore
\begin{eqnarray}\label{LambdaDiff}
\nonumber | \Lambda_{j,G}^{-1} \circ \Lambda_{k,G}(p) -\Lambda_*^{k-j}(p) | &=&\left[\prod_{n=j}^{k-1}(|\lambda_*|+c'_7(\varrho) \omega^{n})-|\lambda_*|^{k-j} \right] |p|\\
\nonumber &=&|\lambda_*|^{k-j} \left[{\rm exp} \left\{ \sum_{n=j}^{k-1}\ln \left(|1+c'_7(\varrho) {\omega^{n}\over |\lambda_*|}\right) \right\}-1 \right] |p|\\
\nonumber &=&|\lambda_*|^{k-j} \left[{\rm exp} \left\{ c_7''(\varrho) \omega^j { 1-\omega^{k-j} \over 1-\omega } \right\}-1 \right] |p|\\
 &\le& c_7(\varrho)\,|\lambda_*|^{k-j} \omega^j |p|,
\end{eqnarray}
where $c_7(\varrho)$ and $c_7''(\varrho)$ are some constants. This, together with $(\ref{unif_converge_omega})$ implies the following bound for any $p^*$ in $\cC_*$ and $p=H_{G_k}(p^*)$ and  all $j<k$
\begin{eqnarray}\label{LambdaDiff*}
\nonumber | \Lambda_{j,G}^{-1}( \Lambda_{k,G}(p)) -\Lambda_*^{k-j}(p^*) | &\le& |\lambda_*|^{k-j}|p-p^*| +c_7(\varrho) \omega^{j} |\lambda_*|^{k-j}|p| \\
\nonumber &\le& C(\varrho) \, |\lambda_*|^{k-j} \omega^{k} +c_8(\varrho) \omega^{j} |\lambda_*|^{k-j},
\end{eqnarray}
where $c_8(\varrho)=c_7(\varrho) {\rm diam}(\{C_G \cup (0,0)\}).$

Next, suppose that $q$ is the smallest integer such that $j_q \ne \hat{j}_q$ and $j_l=\hat{j}_l$, $l < q$. For definitiveness, suppose $\hat{j}_q>j_q$. We expand, as before,
\begin{equation}\label{Giexp}
 p_{k,i} =  \TisG{i}{1}{q-1} \circ \left( \Lambda_{j_{q-1},G}^{-1} \circ  \Lambda_{j_q,G} \right) \circ \left[ G_{j_q} \circ \TisG{i}{q+1}{m} \circ \Lambda_{j_m,G}^{-1} \circ \Lambda_{k,G}\right](p),
\end{equation}
and similarly for $G^{\hat{i}}_*$. Our immediate goal will be to show that $\arrowvert p_{k,i}-s_{k,\hat{i}} \arrowvert$ and  $\arrowvert p_{k,i}^*-s^*_{k,\hat{i}}\arrowvert$ are commensurate. To this end we will show that the distances between the images of points $p$,$s$ and  $p^*$,$s^*$ under the consecutive application of the three maps  $ \TisG{i}{1}{q-1}$, $(\ldots)$ and $\{\ldots\}$ in $(\ref{Giexp})$ stay commensurate. We will perform this in three steps.

\medskip

\noindent {\it \underline{Step (1)}}. Both 
$$\bar{p}_q\equiv
\left[ G_{j_q} \circ \TisG{i}{q+1}{m} \circ \Lambda_{j_m,G}^{-1} \circ \Lambda_{k,G}\right](p) \quad {\rm and} \quad \bar{p}^*_q \equiv \left[ G_* \circ \Tis{i}{q+1}{m} \circ \Lambda^{k-j_m}_*\right](p^*)
$$ 
lie in $\cE_2 \cup \cE_4$. We use $(\ref{LambdaDiff*})$ in the following bound
\begin{eqnarray}\label{pqDiff}
\nonumber |\bar{p}_q-\bar{p}^*_q  |&\le&\left|  G_* \circ \Tis{i}{q+1}{m}  \circ \Lambda_{j_m,G}^{-1}  \circ \Lambda_{k,G}(p) - G_* \circ \Tis{i}{q+1}{m} \circ  \Lambda^{k-j_m}_*(p^*) \right| \\
\nonumber &+&\left|  G_{j_q} \circ \TisG{i}{q+1}{m}  \circ \Lambda_{j_m,G}^{-1} \circ \Lambda_{k,G}(p)-   G_* \circ \Tis{i}{q+1}{m} \circ \Lambda_{j_m,G}^{-1} \circ \Lambda_{k,G}(p) \right| \\
\nonumber &\le &\| D \bG\|_{\cE_1 \cup \cE_3}  A^{m-q} |\lambda_-|^{j_m-j_{q} -(m-q)} \times \\
\nonumber & \times & \left[C(\varrho) \, |\lambda_*|^{k-j_m} \omega^{k} +c_8(\varrho) \omega^{j_m} |\lambda_*|^{k-j_m}\right]+c_9(\varrho) \omega^{j_q}\\
 &\le& c_{10}(\varrho) \,\omega^{j_q}.
\end{eqnarray}
Similarly, $\bar{s}_q$ and $\bar{s}_q^*$ are in $\cE_2 \cup \cE_4$, and $|\bar{s}_q-\bar{s}^*_q  | \le c_{10}(\varrho) \, \omega^{\hat{j}_q}$.

\medskip

\noindent{\it \underline{Step (2)}}. Next, denote,

\medskip

\begin{tabular}{l l l}
$\hat{p}_q \equiv \Lambda_{j_{q-1},G}^{-1} \circ  \Lambda_{j_q,G} (\bar{p}_q)$, & $\hat{p}_q^* \equiv \Lambda^{j_q-j_{q-1}}_* (\bar{p}_q^*)$, & $\tilde{p}_q \equiv \Lambda_*^{j_q-j_{q-1}}(\bar{p}_q)$,\\
$\hat{s}_q \equiv \Lambda_{j_{q-1},G}^{-1} \circ  \Lambda_{\hat{j}_q,G} (\bar{s}_q)$, & $\hat{s}_q^* \equiv \Lambda^{\hat{j}_q-j_{q-1}}_* (\bar{s}_q^*)$, & $\tilde{s}_q \equiv \Lambda_*^{\hat{j}_q-j_{q-1}}(\bar{s}_q)$.
\end{tabular}

\medskip

 According to  Lemma $\ref{separation}$ the sets $\cE_2$ and $\cE_4$ are horizontally separated from their rescalings, therefore,  
\begin{equation}\label{lower_bound}
\arrowvert \bar{p}_q-\Lambda_{j_{q},G}^{-1} \circ \Lambda_{\hat{j}_{q},G}(\bar{s}_q) \arrowvert \ge \inf_{p \in \cE_2 \cup \cE_4
}|\cP_x p| -|\lambda_-| \,  \sup_{p \in \cE_2 \cup \cE_4}|\cP_x p|  \equiv {\rm \delta_1}.
\end{equation}
Clearly, there is also a constant $\delta_2$, such that 
$$\delta_2 > \arrowvert \bar{p}_q-\Lambda_{j_{q},G}^{-1} \circ \Lambda_{\hat{j}_{q},G}(\bar{s}_q) \arrowvert.$$

One can now use a computation similar to $(\ref{LambdaDiff})$, to show that there exists a $c_{11}(\varrho)$, such that 
$$\arrowvert \hat{p}_q - \hat{s}_q \arrowvert = \left| \Lambda_{j_{q-1},G}^{-1} \circ  \Lambda_{j_q,G}\left(\bar{p}_q-\Lambda_{j_{q},G}^{-1} \circ \Lambda_{\hat{j}_{q},G}(\bar{s}_q)\right) \right|$$ 
satisfies 
\begin{equation}\label{commens_bound}
\delta_2  (1+c_{11}(\varrho)\,\omega^{j_{q-1}}) |\lambda_*|^{j_q-j_{q-1}} \ge\arrowvert \hat{p}_q - \hat{s}_q \arrowvert \ge  \delta_1  (1-c_{11}(\varrho)\,\omega^{j_{q-1}}) |\lambda_*|^{j_q-j_{q-1}},  
\end{equation}
and similarly for $\arrowvert \hat{p}_q^* - \hat{s}_q^* \arrowvert$. We use estimates $(\ref{LambdaDiff})$ and  $(\ref{pqDiff})$ to compare these two distances: 
\begin{eqnarray}\label{long_comp}
\nonumber \arrowvert \hat{p}_q - \hat{s}_q \arrowvert  &\le& \arrowvert \hat{p}_q^* - \hat{s}_q^* \arrowvert+\arrowvert \hat{p}_q - \tilde{p}_q  \arrowvert+\arrowvert \tilde{p}_q\!-\! \hat{p}_q^*  \arrowvert +\arrowvert \hat{s}_q - \tilde{s}_q  \arrowvert +\arrowvert \tilde{s}_q - \hat{s}_q^*  \arrowvert \\
\nonumber &\le& \arrowvert \hat{p}_q^* -\hat{s}_q^* \arrowvert+ c_7(\varrho) |\lambda_*|^{j_q-j_{q-1}} \omega^{j_{q-1}}|\bar{p}_q|+|\lambda_*|^{j_q-j_{q-1}} |\bar{p}_q^*-\bar{p}_q|\\
\nonumber &+&c_7(\varrho) |\lambda_*|^{\hat{j}_q-j_{q-1}} \omega^{j_{q-1}}|\bar{s}_q|+|\lambda_*|^{\hat{j}_q-j_{q-1}} |\bar{s}_q^*-\bar{s}_q|\\
\nonumber &\le& \arrowvert \hat{p}_q^* - \hat{s}_q^* \arrowvert+ c_7(\varrho) | \lambda_*|^{j_q-j_{q-1}}  \omega^{j_{q-1}}(|\bar{p}_q|+|\bar{s}_q|) \\ 
\nonumber &+&2 \, c_{10}(\varrho) |\lambda_*|^{j_q-j_{q-1}} \omega^{j_q}
\end{eqnarray}
Therefore, there exists a $c_{12}(\varrho)$, such that
$$ \arrowvert \hat{p}_q -\hat{s}_q \arrowvert \le \arrowvert \hat{p}_q^* - \hat{s}_q^* \arrowvert+ c_{12}(\varrho) |\lambda_*|^{j_q-j_{q-1}} \, \omega^{j_{q-1}},$$
and $\varrho$ can be chosen sufficiently small, so that, for instance, 
$$c_{12}(\varrho) < {1 / 2} \cdot \delta_1  (1-c_{11}(\varrho)\,\omega^{j_{q-1}}),$$ then
$${ \arrowvert \hat{p}_q - \hat{s}_q \arrowvert \over \arrowvert \hat{p}_q^* - \hat{s}_q^* \arrowvert} \le 1+{ c_{12}(\varrho) |\lambda_*|^{j_q-j_{q-1}} \, \omega^{j_{q-1}} \over  \delta_1  (1\!-\!c_{11}(\varrho)\,\omega^{j_{q-1}}) |\lambda_*|^{j_q-j_{q-1}}} \le {3 \over 2}.$$

One can use a similar argument to show that $\varrho$ can be chosen sufficiently small, so that $ \arrowvert \hat{p}_q^* - \hat{s}_q^* \arrowvert / \arrowvert \hat{p}_q - \hat{s}_q \arrowvert$ is also bounded from above by a constant. In particular, if $q=1$, then 
$\arrowvert \hat{p}_q^* - \hat{s}_q^* \arrowvert=|p_{k,i}^*-s_{k,\hat{i}}^*|$ and $\arrowvert \hat{p}_q - \hat{s}_q \arrowvert=|p_{k,i}-s_{k,\hat{i}}|$, and 
$$ \arrowvert\iHkG(p_{k,i}^*)-\!\!\phantom{a}_{\hat{i}}H_{k,G}(s_{k,\hat{i}}^*)\arrowvert= |p_{k,i}-s_{k,\hat{i}}| \asymp |p_{k,i}^*-s_{k,\hat{i}}^*|.$$

\medskip 

\noindent{\it \underline{Step (3)}}. Suppose that $q>1$. We will now demonstrate that 
\begin{equation}
\arrowvert \Tis{i}{1}{q-1}(\hat{p}_q)-\Tis{\hat{i}}{1}{q-1}(\hat{s}_q)\arrowvert \asymp |p_{k,i}-s_{k,\hat{i}}|.
\end{equation}

 Denote for brevity
$${\rm Diff}_l= \left( \TisG{i}{j_{l+1}}{j_{q-1}}(\hat{p}_q)-\TisG{\hat{i}}{j_{l+1}}{j_{q-1}}(\hat{s}_q)\right)-\left( \Tis{i}{j_{l+1}}{j_{q-1}}(\hat{p}_q)-\Tis{\hat{i}}{j_{l+1}}{j_{q-1}}(\hat{s}_q)\right).$$

First, recall that $G_*-G_{j_m}=O(\omega^{j_m})$, and so is the difference of derivatives at any point contained in a compact subset of the domain $\cD_3$.  A straightforward calculation gives for any such point $t$:
\begin{eqnarray}
\label{inv1}D T_{j_m,j_{m-1},G}(t) \cdot D T_{j_m-j_{m-1}}(t)^{-1}&=&{\bf O}(\omega^{j_{m-1}}),\\
\label{inv2}D T_{j_1,j_{0},G}(t) \cdot D T_{j_1}(t)^{-1}&=&{\bf O}(\varrho),
\end{eqnarray}
where ${\bf O}(\delta)$ signifies a $2 \times 2$ matrix whose elements are $O(\delta)$.

Next, 
\begin{eqnarray}
\nonumber {\rm Diff}_{q-2}&=&\left( D \TisG{i}{j_{q-1}}{j_{q-1}}(t^G_{q-2}) - D \Tis{i}{j_{q-1}}{j_{q-1}}(t^*_{q-2}) \right) \cdot (\hat{p}_q -\hat{s}_q)\\
\nonumber &=& \left( D T_{j_{q-1},j_{q-2},G}(t^G_{q-2}) - D T_{j_{q-1}-j_{q-2}}(t^*_{q-2}) \right) \cdot (\hat{p}_q -\hat{s}_q),
\end{eqnarray}
where the derivatives are evaluated at some points $t^G_{q-2}$ and $t^*_{q-2}$  on the line segment between $\hat{p}_q$ and $\hat{s}_q$ (Mean Value Theorem). Again, it is straightforward to demonstrate that $t^G_{q-2}-t^*_{q-2}=O(\omega^{j_{q-2}})$, therefore
\begin{eqnarray}
\nonumber{\rm Diff}_{q-2}&=& \left(D \TisG{i}{j_{q-1}}{j_{q-1}}(t^G_{q-2}) \cdot D \Tis{i}{j_{q-1}}{j_{q-1}}(t^*_{q-2})^{-1} -I\right) \cdot  D \Tis{i}{j_{q-1}}{j_{q-1}}(t^*_{q-2}) \cdot (\hat{p}_q -\hat{s}_q)\\
\nonumber &=&{\bf O}(\omega^{j_{q-2}}) \cdot  D \Tis{i}{j_{q-1}}{j_{q-1}}(t^*_{q-2}) \cdot (\hat{p}_q -\hat{s}_q)\\
\label{Diffq_2} &=&{\bf O}(\omega^{j_{q-2}}) \cdot  \left( \Tis{i}{j_{q-1}}{j_{q-1}}(\hat{p}_q)-\Tis{\hat{i}}{j_{q-1}}{j_{q-1}}(\hat{s}_q)\right).
\end{eqnarray}

Assume,  that for some $2 \le  l<q-2$,
$${\rm Diff}_{l}={\bf O}(\omega^{j_{l}}) \cdot \left( \Tis{i}{j_{l+1}}{j_{q-1}}(\hat{p}_q)-\Tis{\hat{i}}{j_{l+1}}{j_{q-1}}(\hat{s}_q)\right).$$
Then
\begin{eqnarray}
\nonumber{\rm Diff}_{l-1}&=&\left( T_{j_{l},j_{l-1},G} \circ \TisG{i}{j_{l+1}}{j_{q-1}}(\hat{p}_q)-T_{j_{l},j_{l-1},G} \circ \TisG{\hat{i}}{j_{l+1}}{j_{q-1}}(\hat{s}_q)\right)\\
\nonumber &\phantom{=}&-\left( T_{j_{l},j_{l-1}} \circ \Tis{i}{j_{l+1}}{j_{q-1}}(\hat{p}_q)-T_{j_{l},j_{l-1}} \circ \Tis{\hat{i}}{j_{l+1}}{j_{q-1}}(\hat{s}_q)\right)\\
\nonumber&=&D T_{j_{l},j_{l-1},G}  (\tilde{t}^G_{l-1}) \cdot \left(  \TisG{i}{j_{l+1}}{j_{q-1}}(\hat{p}_q)- \TisG{\hat{i}}{j_{l+1}}{j_{q-1}}(\hat{s}_q)\right)\\
\nonumber &\phantom{=}&- D T_{j_{l},j_{l-1}}  (\tilde{t}^*_{l-1}) \cdot \left(\Tis{i}{j_{l}+1}{j_{q-1}}(\hat{p}_q)- \Tis{\hat{i}}{j_{l+1}}{j_{q-1}}(\hat{s}_q)\right)\\
\nonumber&=&\left\{ D T_{j_{l},j_{l-1},G} (\tilde{t}^G_{l-1}) \cdot \left(  \Tis{i}{j_{l+1}}{j_{q-1}}(\hat{p}_q)- \Tis{\hat{i}}{j_{l+1}}{j_{q-1}}(\hat{s}_q)\right)\right.\\
\nonumber &\phantom{=}&\left.- D T_{j_{l},j_{l-1}} (\tilde{t}^*_{l-1}) \cdot \left(  \Tis{i}{j_{l+1}}{j_{q-1}}(\hat{p}_q)- \Tis{\hat{i}}{j_{l+1}}{j_{q-1}}(\hat{s}_q)\right) \right\}\\
\nonumber&\phantom{=}&+\left[D T_{j_{l},j_{l-1},G} (\tilde{t}^G_{l-1}) \cdot \left(  \TisG{i}{j_{l+1}}{j_{q-1}}(\hat{p}_q)- \TisG{\hat{i}}{j_{l+1}}{j_{q-1}}(\hat{s}_q)\right)\right.\\
\nonumber &\phantom{=}&\left.- D T_{j_{l},j_{l-1},G} (\tilde{t}^G_{l-1}) \cdot \left(\Tis{i}{j_{l+1}}{j_{q-1}}(\hat{p}_q)- \Tis{\hat{i}}{j_{l+1}}{j_{q-1}}(\hat{s}_q)\right)\right],
\end{eqnarray}
where $\tilde{t}^G_{l-1}$ is some point on the line segment connecting $\TisG{i}{j_{l+1}}{j_{q-1}}(\hat{p}_q)$ and $\TisG{\hat{i}}{j_{l+1}}{j_{q-1}}(\hat{s}_q)$, while  $\tilde{t}^*_{l-1}$ is a point on the line segment between $\Tis{i}{j_{l+1}}{j_{q-1}}(\hat{p}_q)$ and $\Tis{\hat{i}}{j_{l+1}}{j_{q-1}}(\hat{s}_q)$.''

 We treat expression $\{\ldots\}$ and $[ \ldots]$ separately. The first one is worked out similarly to $(\ref{Diffq_2})$:
\begin{eqnarray}
\nonumber\left\{\ldots \right\}&=&{\bf O}(\omega^{j_{l-1}}) \cdot D T_{j_{l},j_{l-1}} (\tilde{t}^*_{l-1})  \cdot \left(  \Tis{i}{j_{l+1}}{j_{q-1}}(\hat{p}_q)- \Tis{\hat{i}}{j_{l+1}}{j_{q-1}}(\hat{s}_q)\right)\\
\nonumber&=&{\bf O}(\omega^{j_{l-1}})  \cdot \left(  \Tis{i}{j_{l}}{j_{q-1}}(\hat{p}_q)- \Tis{\hat{i}}{j_{l}}{j_{q-1}}(\hat{s}_q)\right).
\end{eqnarray}

\begin{eqnarray}
\nonumber\left[\ldots \right]&=&D T_{j_{l},j_{l-1},G} (\tilde{t}^G_{l-1})  \cdot {\rm Diff_l}=(I+{\bf O}(\omega^{j_{l-1}}) ) \cdot D T_{j_{l},j_{l-1}} (\tilde{t}^*_{l-1})  \cdot {\rm Diff_l}\\
\nonumber&=&(I+{\bf O}(\omega^{j_{l-1}}))  \cdot D T_{j_{l},j_{l-1}} (\tilde{t}^*_{l-1}) \cdot {\bf O}(\omega^{j_{l}}) \cdot  \left( \Tis{i}{j_{l+1}}{j_{q-1}}(\hat{p}_q)-\Tis{\hat{i}}{j_{l+1}}{j_{q-1}}(\hat{s}_q)\right) \\
\nonumber&=&(I+{\bf O}(\omega^{j_{l-1}})) \cdot {\bf O}(\omega^{j_{l}}) \cdot  D T_{j_{l},j_{l-1}} (\tilde{t}^*_{l-1}) \cdot  \left( \Tis{i}{j_{l+1}}{j_{q-1}}(\hat{p}_q)-\Tis{\hat{i}}{j_{l+1}}{j_{q-1}}(\hat{s}_q)\right) \\
\nonumber&=&(I+{\bf O}(\omega^{j_{l-1}}))  \cdot {\bf O}(\omega^{j_{l}}) \cdot \left(  \Tis{i}{j_{l}}{j_{q-1}} (\hat{p}_q)- \Tis{\hat{i}}{j_{l}}{j_{q-1}}(\hat{s}_q)\right). 
\end{eqnarray}

Therefore, 
$${\rm Diff}_{l-1}={\bf O}(\omega^{j_{l-1}}) \cdot  \left( \Tis{i}{j_{l}}{j_{q-1}}(\hat{p}_q)-\Tis{\hat{i}}{j_{l}}{j_{q-1}}(\hat{s}_q)\right),$$
which completes the induction. In particular, when $l=1$, we obtain
$$\left( \TisG{i}{j_{1}}{j_{q-1}}(\hat{p}_q)-\TisG{\hat{i}}{j_{1}}{j_{q-1}}(\hat{s}_q)\right)=(I+{\bf O}(\varrho)) \cdot \left(  \Tis{i}{j_{1}}{j_{q-1}}(\hat{p}_q)- \Tis{\hat{i}}{j_{1}}{j_{q-1}}(\hat{s}_q)\right),$$
where we have used $(\ref{inv2})$, or
\begin{equation}\label{Gcommens}
 \arrowvert \TisG{i}{1}{q-1}(\hat{p}_q)-\TisG{\hat{i}}{1}{q-1}(\hat{s}_q)\arrowvert \le (1+c_{13}(\varrho)) \,\arrowvert \Tis{i}{1}{q-1}(\hat{p}_q)-\Tis{\hat{i}}{1}{q-1}(\hat{s}_q)\arrowvert.
\end{equation}
for some ``constant'' $c_{13}$, $\lim_{\varrho \rightarrow 0}c_{13}(\varrho)=0$.

On the other hand, if we interchange the roles of $G$ and $G_*$ in the computation above, we get
$$ \arrowvert \Tis{i}{1}{q-1}(\hat{p}_q)-\Tis{\hat{i}}{1}{q-1}(\hat{s}_q)\arrowvert \le (1+c_{13}(\varrho)) \,\arrowvert \TisG{i}{1}{q-1}(\hat{p}_q)-\TisG{\hat{i}}{1}{q-1}(\hat{s}_q)\arrowvert.$$

\noindent{\it \underline{Step (4)}}.  Finally, we  demonstrate that 
$$|p_{k,i}-s_{k,\hat{i}}|= \arrowvert \TisG{i}{1}{q-1}(\hat{p}_q)-\TisG{\hat{i}}{1}{q-1}(\hat{s}_q)\arrowvert$$ 
and 
$$|p_{k,i}^*-s_{k,\hat{i}}^*|=|\Tis{i}{1}{q-1}(\hat{p}_q^*)-\Tis{\hat{i}}{1}{q-1}(\hat{s}_q^*)|$$
are commensurate.

We will compare 
$$I_1 \equiv \arrowvert \Tis{i}{1}{q-1}(\hat{p}^*_q)-\Tis{\hat{i}}{1}{q-1}(\hat{s}^*_q)\arrowvert$$ 
and 
$$I_2 \equiv \arrowvert \Tis{i}{1}{q-1}(\hat{p}_q)-\Tis{\hat{i}}{1}{q-1}(\hat{s}_q)\arrowvert.$$
Since $\Tis{i}{1}{q-1}=\Tis{\hat{i}}{1}{q-1}$,
\begin{equation}\label{upper_lower}
 b^{q-1} \mu_*^{j_{q-1}-(q-1)} |\hat{s}^*_q-\hat{p}^*_q|  \le   I_1 \le A^{q-1} |\lambda_*|^{j_{q-1}-(q-1)} |\hat{s}^*_q-\hat{p}^*_q|,
\end{equation}
and 
\begin{eqnarray}
\nonumber I_2 &\le& I_1+  \arrowvert \Tis{i}{1}{q-1}(\hat{p}^*_q)-\Tis{i}{1}{q-1}(\hat{p}_q)\arrowvert + \arrowvert \Tis{\hat{i}}{1}{q-1}(\hat{s}^*_q)-\Tis{\hat{i}}{1}{q-1}(\hat{s}_q)\arrowvert\\
\nonumber &\le& I_1+  A^{q-1} |\lambda_*|^{j_{q-1}-(q-1)} \left( |\hat{p}^*_q-\hat{p}_q| + |\hat{s}^*_q-\hat{s}_q)| \right)\\
\nonumber  &\le& I_1+  c_{12}(\varrho) \, A^{q-1} |\lambda_*|^{j_{q}-(q-1)} \, \omega^{j_{q-1}}.
\end{eqnarray}

These two estimates put together with the estimate $(\ref{commens_bound})$ result in the following bound:
\begin{eqnarray}
\nonumber { I_2 \over I_1}  &\le& 1+  c_{12}(\varrho) { A^{q-1} |\lambda_*|^{j_{q}-(q-1)}  \, \omega^{j_{q-1}} \over  b^{q-1} \mu_*^{j_{q-1}-(q-1)} |\hat{p}^*_q-\hat{s}^*_q|  } \omega^{j_{q-1}}  \\
\nonumber &\le&  1+   {c_{12}(\varrho) \over  \delta_1  (1-c_{11}(\varrho)\,\omega^{j_{q-1}}) } { A^{q-1} |\lambda_*|^{j_{q}-(q-1)} \over  b^{q-1} \mu_*^{j_{q-1}-(q-1)} |\lambda_*|^{j_q-j_{q-1}} } \omega^{j_{q-1}} \\
\nonumber &\le&  1+  {c_{12}(\varrho) \over  \delta_1  (1-c_{11}(\varrho)\,\omega^{j_{q-1}}) } \left( { A \mu_* \over b |\lambda_*|} \right)^{q-1} \left[ {|\lambda_*| \over \mu_* } \right] ^{j_{q-1}}  \omega^{j_{q-1}} \\
\nonumber &\le& \left\{  1+  {c_{12}(\varrho) \over  \delta_1  (1-c_{11}(\varrho)\,\omega^{j_{q-1}}) } \left( { A \mu_* \over b |\lambda_*|} \right)^{j_{q-1}} \left[ {|\lambda_*| \over \mu_* } \right] ^{j_{q-1}}  \omega^{j_{q-1}}, \quad { A \mu_* \over b |\lambda_*|} >1  \atop    1+  {c_{12}(\varrho) \over  \delta_1  (1-c_{11}(\varrho)\,\omega^{j_{q-1}}) } \left[ {|\lambda_*| \over \mu_* } \right] ^{j_{q-1}}  \omega^{j_{q-1}}, \quad { A \mu_* \over b |\lambda_*|} \le 1   \right.\\
\label{Imp_Bound} &\le&  1+  {c_{12}(\varrho) \over  \delta_1  (1-c_{11}(\varrho)\,\omega^{j_{q-1}}) }  \max \left\{{|\lambda_*| \over \mu_*},  {A \over b} \right\}^{j_{q-1}}  \omega^{j_{q-1}}.
\end{eqnarray}
and similarly for $I_1/I_2$. Therefore, if 
$$\omega \le \min \left\{  {\mu_* \over |\lambda_*|},  {b \over A} \right\},$$ then
\begin{equation}\label{bi-Lips}
 \arrowvert\iHkG(p_{k,i}^*)-\!\!\phantom{a}_{\hat{i}}H_{k,G}(s_{k,\hat{i}}^*)\arrowvert= |p_{k,i}-s_{k,\hat{i}}| \asymp I_2 \asymp I_1 = |p_{k,i}^*-s_{k,\hat{i}}^*|.
\end{equation}
\end{proof}

\begin{remark}
We would like to emphasize that the commensurability property $(\ref{bi-Lips})$ holds only for $i \ne \hat{i}$, and therefore does not imply that the hyperbolic sets $\cC^k_G$ and $\cC^k_*$ are bi-Lipschitz conjugate. In case $i=\hat{i}$ a positive lower bound $(\ref{lower_bound})$ does not exist, which would invalidate the arguments that follow.
\end{remark}







\medskip

\section{Some concluding remarks}

We have demonstrated that the Hausdorff dimension of the stable set for the maps $F$ in the subset $\bW_\omega(\varrho)$ of the infinitely renormalizable maps is independent of $F$, and that the stable dynamics for two infinitely renormalizable maps in  $\bW_\omega(\varrho)$ is bi-Lipschitz-conjugate.  This is quite weaker than the corresponding  result about the invariance of the Hausdorff dimension of the Feigenbaum attractor for all infinitely renormalizable unimodal maps (see \cite{Pal,Rand,McM1,dMP}). On the other hand, it does demonstrate that one should expect at least some kind of rigidity of invariant sets for infinitely renormalizable maps in conservative dynamics  ---  rigidity which was absent in dissipative maps (see \cite{dCLM}).

Our proof of the bi-Lipschitz property of the conjugacy between stable sets $\cC_G^\infty$ and $\cC_{\tilde{G}}^\infty$ balances two phenomena that, in a sense, work against each other: convergence of renormalizations of maps $G \in \bW(\varrho)$  versus the fact that the rates of contraction of distances in different directions by maps $\Lambda_F$ and $\Lambda_F \circ F$ are essentially different. A careful look at the proof shows that  the bi-Lipschitz property is achieved if the convergence rate $\nu$ is sufficiently small to ``counteract'' the relative size of contractions. However, this is not the case with the upper bound $(\ref{contr_rate})$ on $\nu$ at hand. Although this upper bound is by no means sharp, it does indicate that one might need to choose a submanifold $\bW_\omega(\varrho)$ of $\bW(\varrho)$ on which the convergence rate is smaller.



Another obvious issue for investigation is whether the bi-Lipschitz conjugacy of the stable sets extends to their neighborhood as a $C^{1+\epsilon}$ map. Again, this is the case for the conjugacies between attractors of the unimodal maps (see \cite{Rand,McM1,dMP}), and it is not for very dissipative maps where, as we have already mentioned, the regularity of the conjugacy of attractors for two maps $F$ and $\tilde{F}$ has a definite upper bound $(\ref{reg_bound})$.

\medskip

\section{Acknowledgments}

The authors would like to thank Hans Koch for his many useful 
insights into period-doubling in area-preserving maps, as well as for 
his great help with understanding the original computer assisted proof 
\cite{EKW2}.




\end{document}